\documentclass[12pt,reqno]{amsart}

\usepackage{amsmath}
\usepackage{amsthm}
\usepackage{cite}
\usepackage[all]{xy}
\usepackage{latexsym}
\usepackage{rotating}
\usepackage{amsfonts}
\usepackage{amssymb}
\usepackage{amsmath}
\usepackage{hyperref,url}

\newtheorem{theorem}{Theorem}[section]

\newtheorem{remark}[theorem]{Remark}

\newtheorem{example}[theorem]{Example}

\title[Rank 1 character varieties]{Rank 1 character varieties of finitely presented groups.}

\date{\today}

\author[C. Ashley]{Caleb Ashley}

\address{Department of Mathematics, University of Michigan, 530 Church Street Ann Arbor, MI 48109, USA}

\email{cjashley@umich.edu}

\author[JP. Burelle]{Jean-Philippe Burelle}

\address{Department of Mathematics, University of Maryland, 4176 Campus Drive, College Park, MD 20742, USA}

\email{jburelle@math.umd.edu}

\author[S. Lawton]{Sean Lawton}

\address{Department of Mathematical Sciences, George Mason University, 4400 University Drive, Fairfax, Virginia 22030, USA}

\email{slawton3@gmu.edu}

\newcommand{\C}{\mathbb{C}}
\newcommand{\Z}{\mathbb{Z}}

\newcommand{\R}{\mathbb{R}}
\newcommand{\SL}{\mathsf{SL}}
\newcommand{\PSL}{\mathsf{PSL}}

\newcommand{\hm}{\mathrm{Hom}}
\newcommand{\xb}{\mathbf{X}}
\newcommand{\yb}{\mathbf{Y}}
\newcommand{\zb}{\mathbf{Z}}
\newcommand{\wb}{\mathbf{W}}
\newcommand{\ub}{\mathbf{U}}
\newcommand{\vb}{\mathbf{V}}
\newcommand{\ab}{\mathbf{A}}
\newcommand{\bb}{\mathbf{B}}
\newcommand{\cb}{\mathbf{C}}
\newcommand{\db}{\mathbf{D}}
\newcommand{\tr}{\mathrm{tr}}
\newcommand{\X}{\mathfrak{X}}
\newcommand{\Y}{\mathfrak{Y}}
\newcommand{\id}{\mathbf{1}}
\newcommand{\quot}{/\!\!/}

\begin{document}

\begin{abstract}
Let $\X(\Gamma,G)$ be the $G$-character variety of $\Gamma$ where $G$ is a rank 1 complex affine algebraic group and $\Gamma$ is a finitely presentable discrete group. We describe an algorithm, which we implement in {\it Mathematica}, {\it SageMath}, and in Python, that takes a finite presentation for $\Gamma$ and produces a finite presentation of the coordinate ring of $\X(\Gamma,G)$.  We also provide a new description of the defining relations and local parameters of the coordinate ring when $\Gamma$ is free.  Although the theorems used to create the algorithm are not new, we hope that as a well-referenced exposition with a companion computer program it will be useful for computation and experimentation with these moduli spaces.
\end{abstract}

\dedicatory{To William Goldman on the occasion of his 60th birthday.}

\maketitle

\section{Introduction} 

Numerous applications motivate the study of families of homomorphisms from a discrete group to a Lie group.  Perhaps the most impactful (and so compelling) is the relationship in differential geometry to flat bundles (see \cite[Page 163]{thurston}).  Given a manifold $M$ with a base point $m\in M$ and a Lie group $G$, one can construct a pointed principal $G$-bundle over $M$ as follows.  Let $\tilde{M}$ be the universal cover of $M$ and take a homomorphism $\rho:\pi_1(M,m)\to G$.  The group $\Gamma:=\pi_1(M,m)$ acts on $\tilde{M}\times G$ by deck transformations in the first factor and by $\rho$ in the second.  Precisely, given $(x,g)\in \tilde{M}\times G$ and $\gamma \in \Gamma$ the action is $\gamma\cdot (x,g)=(\delta_\gamma(x),\rho(\gamma)g)$ where $\delta_\gamma$ is the corresponding deck transformation in the deck group $\Delta\cong \Gamma$.  This action is free since it is free on the first factor and so we obtain a principal $G$-bundle $E_\rho\to M$ by considering the quotient by $\Gamma$ of the projection $\tilde{M}\times G\to \tilde{M}$ since $\tilde{M}/\Gamma\cong M$.  The transition maps are locally constant and hence the bundle is flat.  Conversely, any flat principal $G$-bundle $P\to M$ has a holonomy homomorphism $\rho:\Gamma\to M$, once a base point is stipulated, so that $P\cong E_\rho$.  We see that the collection of flat principal $G$-bundles over $(M,m)$ is in bijective correspondence with the set $\mathrm{Hom}(\pi_1(M,m),G)$.  Changing base points corresponds exactly to conjugating homomorphisms by $G$.  Hence the conjugation quotient $\mathrm{Hom}(\pi_1(M,m),G)/G$ is bijectively equivalent to the collection of flat principal $G$-bundles over $M$ (without dependence on the base point).

Using homomorphisms to tie the geometry of a Lie group $G$ to the topology of a manifold $M$ via its fundamental group accounts for many, if not all, of the ways that spaces of homomorphisms appear in geometry, topology, and mathematical physics.  Examples occur in the study of knot invariants \cite{CCGLS}, hyperbolic geometry \cite{CS}, holomorphic vector bundles \cite{NS1}, Yang-Mills connections \cite{AB,  Hi}, Higgs bundles \cite{Simpson1, Simpson2}, flat principal bundles \cite{Borel-Friedman-Morgan}, the Geometric Langlands Program \cite{HaTh,KW}, and supersymmetry \cite{Kac-Smilga, Wi}.

Another important example where spaces of homomorphisms are prominent is the study of locally homogeneous structures on manifolds \cite{G6, GoldmanICM}.  Here the problem is, for a given topology on a compact manifold $M$ and a model geometry $X = G/H$ where $G$ is a Lie group, to classify all the possible ways of introducing the local geometry of $X$ onto the topology of $M$. The Ehresmann-Thurston Theorem shows that the moduli space of such marked $(G,X)$-structures is locally homeomorphic to the space of conjugacy classes of $\mathrm{Hom}(\pi_1(M,m),G)$ with the quotient topology.

Here is a simple but illustrative example of a space of flat bundles.  If $M$ is an annulus then $\Gamma\cong \Z$ and $\hm(\Gamma,G)\cong G$ by the evaluation mapping sending $1\mapsto g$ for any $g\in G$.  Say $G=\SL_2(\C)$.  Then each conjugacy class of a diagonalizable $g\in G$ is determined by its characteristic polynomial $\lambda^2-t\lambda+1$ where $t=\tr(g)$. There are only two classes of non-diagonalizable matrices: one with trace $2$ and the other with trace $-2$.  So the space of flat $\SL_2(\C)$-bundles over an annulus, with its closed points parametrized by $t$, is $\C\sqcup\{2_0,-2_0\}$ where $\pm 2_0$ are ``double-points" corresponding to the orbits of non-diagonalizable matrices with trace $\pm 2$.  In particular, the space is not $T_1$ since the closure of $\pm 2_0$ is $\{\pm 2_0, \pm 2\}$.

One way to deal with the topological complications illustrated above is to consider an approximation of the quotient space $\hm(\Gamma, G)/G$. Concretely, when $\Gamma$ is finitely presentable (which means there is a finite rank free group $F_r$ and an epimorphism $\varphi:F_r\to \Gamma$ with finitely generated kernel), we can consider: $$\Gamma\cong\langle \gamma_1,...,\gamma_r\ |\ R_1,...,R_s\rangle.$$  There is a map, depending on $\varphi$, $E_\varphi:\hm(\Gamma,G)\to G^r$ given by $E_\varphi(\rho)=(\rho(\gamma_1),...,\rho(\gamma_r))$.  This map, called an {\it evaluation map}, is clearly injective and has as its image the set $\{(g_1,...,g_r)\ |\ R_i(g_1,...,g_r)=\id,\ 1\leq i\leq s\}$ where $\id\in G$ is the identity.  As $G$ is a Lie group the image inherits the subspace topology from $G^r$, and is consequently an analytic subvariety (algebraic if $G$ is algebraic).  One can show that $E_\varphi$ is an embedding when $\hm(\Gamma, G)$ is given the more intrinsic compact-open topology with respect to the discrete topology on $\Gamma$.  This observation shows that any two presentations for $\Gamma$ result in homeomorphic topologies on $\hm(\Gamma, G)$ as analytic varieties.  

Now consider the subspace of {\it polystable} points:  $\hm(\Gamma, G)^*=\{\rho\in \hm(\Gamma, G)\ |\ G\rho=\overline{G\rho}\}$, where $G\rho$ denotes the conjugation orbit of $\rho$ and $\overline{G\rho}$ is the closure of the orbit in the aforementioned topology on $\hm(\Gamma, G)$.  Then the orbit space $\X(\Gamma,G)=\hm(\Gamma, G)^*/G$ is $T_1$, and called the {\it moduli space of $G$-representations of $\Gamma$}.  

In the above example, $\X(\Z,\SL_2(\C))=\C$ since $\pm 2_0$ correspond to non-closed orbits.  

Here are some facts about $\X(\Gamma,G)$.  When $G$ is real reductive then $\X(\Gamma,G)$ is Hausdorff and when $G$ is also real algebraic $\X(\Gamma,G)$ is moreover semi-algebraic (see the discussion in \cite{CFLO1,CFLO2}).  In the special case when $G$ is compact, then $\X(\Gamma,G)=\hm(\Gamma, G)/G$.  In the special case when $G$ is complex affine algebraic and reductive, then $\X(\Gamma,G)$ is homeomorphic to the Geometric Invariant Theory (GIT) quotient $\hm(\Gamma, G)\quot G$ with the Euclidean topology (see \cite[Proposition 2.3]{FlLa4}) and so is naturally identified with the $\C$-points of an affine algebraic set.  Moreover, again when $G$ is reductive and complex algebraic $\X(\Gamma,G)$ is {\it homotopic} to the potentially non-Hausdorff quotient $\hm(\Gamma,G)/G$ (see \cite[Proposition 3.4]{FLR}).  With these cases in mind, we consider $\X(\Gamma,G)$ a reasonable approximation to $\hm(\Gamma, G)/G$ in general.

We say a few more words about the case when $G$ is complex reductive.  In this case, the GIT quotient is $\mathrm{Spec}(\C[\hm(\Gamma, G)]^G)$ where $\C[\hm(\Gamma, G)]$ is the coordinate ring of the affine algebraic set $\hm(\Gamma,G)$ and $\C[\hm(\Gamma, G)]^G$ is the ring of invariant elements in the coordinate ring under the conjugation action $(g\cdot f)(x)=f(g^{-1}xg)$.  

In the case $G=\SL_n(\C)$ the coordinate ring of $\X(\Gamma,G)$ is generated by functions of the form $\tr(\wb)$ where $\wb$ is a word in generic unimodular matrices (we will say more about this in Section \ref{gen}).  This explains why the moduli space of $G$-representations is called ``character variety,'' as it is a variety of characters.  However, as discussed in \cite{LaSi}, for some $G$ this later fact remains true and for others it does not.  On the other hand, the main theorem in \cite{LaSi} says that a finite quotient of the stable locus of $\X(\Gamma,G)$ is always in birationally bijective correspondence to a variety of characters, in particular, they have the same Grothendieck class. 

In this paper we assume $G=\SL_2(\C)$ and $\Gamma$ is any finitely presented group.  We provide an {\it effective} algorithm (that does not rely on elimination ideal methods) to determine the structure of the variety $\X(\Gamma,G)$, and provide programs written in {\it Mathematica}, {\it SageMath}, and Python that implement our algorithm.\footnote{Available at \url{http://math.gmu.edu/~slawton3/trace-identities.nb}, \url{http://math.gmu.edu/~slawton3/Main.sagews}, and \url{http://math.gmu.edu/~slawton3/charvars.py} respectively.}  As far as we know this is the first time this algorithm has been put together in this generality.  The theorems used to construct this algorithm are not new however.  They are an amalgamation of results from Vogt \cite{Vogt}, Horowitz \cite{Horowitz}, Magnus \cite{Magnus}, Gonz{\'a}lez-Acu{\~n}a and Montesinos-Amilibia \cite{GoMo}, Brumfiel and Hilden \cite{BrumHild}, and Drensky \cite{D}.  We put them together to provide a new proof and description of the structure of $\SL_2$-character varieties.  In Section 2 we describe a generating set and local parameters of $\X(\Gamma,G)$, and in Section 3 we describe a generating set for the ideal of defining relations in the aforementioned generators.  In Section 4 we implement our algorithm to compute concrete examples, verifying known results.

\section*{Acknowledgements}
Lawton thanks Chris Manon and Bill Goldman for helpful conversations.  This material is based upon work supported by the National Science Foundation (NSF) under grant number DMS 1321794; the Mathematics Research Communities (MRC) program.  All three authors benefited from our time in Snowbird with the MRC (2011 and 2016) and we are very grateful for the wonderful support they provided.  In particular, special thanks goes to Christine Stevens.  Additionally, Lawton was supported by the Simons Foundation Collaboration grant 245642, and the NSF grant DMS 1309376.  Lastly, we acknowledge support from NSF grants DMS 1107452, 1107263, 1107367 ``RNMS: GEometric structures And Representation varieties" (the GEAR Network).  We thank an anonymous referee for helping improve the paper.

\section{Generators of the Character Variety}\label{gen}

Let $G$ be a reductive group scheme defined over $\Z$.  Then it represents a functor from algebras to groups.  For a finitely presentable group $\Gamma$ there is a natural functor from algebras to sets $\mathcal{R}(\Gamma, G)$ given by $A\mapsto \hm(\Gamma, G(A))$.  As shown in \cite{LubMag} this functor represents a scheme which we denote the same way.  Then by GIT (see \cite{MFK, Se77}) we have an affine scheme over $\Z$ given by $\mathcal{X}(\Gamma, G)=\mathcal{R}(\Gamma, G)\quot G$ whose underlying variety is $\X(\Gamma, G)$.  In other words, $\mathcal{X}(\Gamma, G)/\sqrt{0}\cong \mathfrak{X}(\Gamma, G)$.  By \cite{KaMi, KaMi2} there are in fact non-zero nilpotents in these schemes for some choices of $\Gamma$ and $G$.   See \cite{Si2, Si4} for an alternative description of the scheme structure of $\mathcal{X}(\Gamma,G)$.

When $\Gamma=F_r$ is a rank $r$ free group, then $\hm(\Gamma,G)\cong G^r$ and naturally has no nilpotents.  When $G$ is moreover connected it is also irreducible.  Therefore when $G$ is connected and $\Gamma$ is free $\mathcal{X}(\Gamma,G)$ is reduced and irreducible. By \cite[Proposition 2.6]{Ha} the functor from schemes to varieties is faithful and the scheme structure of $\mathcal{X}(F_r,G)$ is determined by the structure of the variety $\X(F_r,G)$.  In the coming sections we will explicitly write out the structure of the scheme $\mathcal{X}(F_r,\SL_2)$ over $\Z[1/2]$, and the structure of the variety $\X(\Gamma, \SL_2)$ for any finitely presented $\Gamma$.

First we describe a natural relationship between non-commutative algebra and $\SL_2(\C)$-character varieties that was more generally described in \cite{La0} for $\SL_n$. We focus on the $n=2$ case as it is the topic of this treatise.

Consider the polynomial ring $\mathcal{P}_r=\C[x^k_{ij}\ |\ 1\leq i,j\leq 2,\ 1\leq k\leq r]$ in $4r$ indeterminates, and in those terms define {\it generic matrices} as $\xb_k=\left(
\begin{array}{cc}
x^k_{11} & x^k_{12} \\
x^k_{21} & x^k_{22}  \\
\end{array}\right),$ for $1\leq k\leq r$.  

Denote the $\C$-algebra on $\xb_1,...,\xb_r$ with addition given by matrix addition and multiplication given by matrix multiplication by $\mathcal{M}_r=\C[\xb_1,...,\xb_r]$.  Elements in $\mathcal{M}_r$ have the form $\sum_{i\in I}\lambda_i \wb_i$ where $\wb_i$ is a finite product of elements in the letters $\xb_1,...,\xb_r$.  Such a product is called a {\it word}; denote the set of all such words $\mathcal{W}$. There is a natural map $\tr:\mathcal{M}_r\to \mathcal{P}_r$ given by $\sum_{i\in I}\lambda_i \wb_i\mapsto  \sum_{i\in I}\lambda_i \tr(\wb_i).$
Now consider the monoid of $2\times 2$ complex matrices $\mathrm{M}_2(\C)$ and the affine space $\mathrm{M}_2(\C)^r\cong \C^{4r}$.  Then $G=\SL_2(\C)$ acts on $\C^{4r}$ by $g\cdot (A_1,...,A_r)=(gA_1g^{-1},...,gA_rg^{-1})$.  With respect to this action, denote the GIT quotient $\mathfrak{Y}_r=\C^{4r}\quot G$.  By \cite{P1,Ra} the image of $\tr$ generates the coordinate ring of $\mathfrak{Y}_r$, that is, $\langle\tr(\mathcal{M}_r)\rangle=\C[\mathfrak{Y}_r]$.  In particular, this shows that $\C[\mathfrak{Y}_r]=\C[\mathrm{M}_2(\C)^r]^G=\C[\tr(\wb)\ |\ \wb \in \mathcal{W}]$.  By Nagata's solution to Hilbert's 14-th problem \cite{nagata}, we know this algebra is finitely generated since $G$ is reductive.  We refer to the elements $\tr(\wb)$ as {\it trace coordinates}.

We now relate this algebra to $\C[\X(F_r,G)]$. For $\X(F_r,\SL_2(\C))$ we write $\X_r$ to simplify the notation.  Note that $\C[\hm(F_r,\SL_2(\C))]=\mathcal{P}_r/\Delta_r$ where $\Delta_r$ is the ideal generated by 
the $r$ irreducible polynomials $\det(\xb_k)-1=x^{k}_{11}x^{k}_{22}-x^{k}_{12}x^{k}_{21}-1.$

Then $\C[\mathfrak{Y}_r]/\Delta\approx \C[\X_r]$. Otherwise stated, $$\mathcal{P}_r^{\SL_2(\C)}/\Delta\approx \left( \mathcal{P}_r/\Delta \right)^{\SL_2(\C)},$$ which is true because $\SL_2(\C)$ is {\it linearly} reductive and the generators of $\Delta$ are conjugation invariant (see \cite[Lemma 3.5]{Do}).  Let $\overline{x}^k_{ij}$ be the image of $x^{k}_{ij}$ under $\mathcal{P}_r \to \mathcal{P}_r/\Delta$, and define $\overline{\xb_k}=\left(
\begin{array}{cc}
\overline{x}^k_{11} & \overline{x}^k_{12} \\
\overline{x}^k_{21} & \overline{x}^k_{22}  \\
\end{array}\right).$ They are called {\it generic unimodular matrices}.  Let $\overline{\mathcal{W}}$ be the collection of all words in the letters $\overline{\xb_1},...,\overline{\xb_r}$.   The discussion above shows $$\C[\X_r]=\C[\hm(F_r,\SL_2(\C))]^{\SL_2(\C)}=\C[\tr(\wb)\ |\ \wb \in \overline{\mathcal{W}}].$$

\subsection{Free Groups}\label{sec:freegroups}
In this subsection, we start with the infinite generating set of trace coordinates for $\C[\X_r]$ and reduce it to a finite set.  We then show it is {\it unshortenable}, that is, the finite subset of generators is a minimal generating set of trace coordinates. Lastly we note that this unshortenable generating set is in fact minimal among all generating sets for the ring $\C[\X_r]$.  Thereafter, we pick out a maximal algebraically independent subset of these generators which give local parameters.

Indeed, working in $\mathcal{M}_r$, the Cayley-Hamilton (characteristic) equation gives $$\xb^2-\tr(\xb)\xb+\mathrm{det}(\xb)\id=\mathbf{0}.$$  And if we assume $\mathrm{det}(\xb)=1$, as is the case in $\C[\X_r]$, we easily derive $\tr(\xb^{-1})=\tr(\xb)$ and $\tr(\xb^2)=\tr(\xb)^2-2$.  Hence the generators $\tr(\xb^2)$ in $\C[\Y_r]$ project to $\tr(\xb)^2-2$ in $\C[\X_r]$ and so are not needed to generate the ring. This is essentially the only difference between $\C[\Y_r]$ and $\C[\X_r]$.

Minimal generators for $\C[\Y_r]$ were first worked out by Sibirskii in $1968$ \cite{Si68}; see also \cite{LeBruyn,T2}.  Those generators project to an unshortenable generating set for $\C[\X_r]$ (see \cite[Proposition 21]{La3}), and are in fact minimal among all generating sets by \cite[Proposition 3.3]{GL}.  

From \cite{Vogt} we have:
\begin{eqnarray}\nonumber
\tr(\xb\yb\zb\wb)&=&(1/2)(\tr(\xb)\tr(\yb)\tr(\zb)\tr(\wb)+\tr(\xb)\tr(\yb\zb\wb)\\ \nonumber
&+&\tr(\yb)\tr(\xb\zb\wb)+\tr(\zb)\tr(\xb\yb\wb)+\tr(\wb)\tr(\xb\yb\zb)\\ \nonumber
&-&\tr(\xb\zb)\tr(\yb\wb)+\tr(\xb\wb)\tr(\yb\zb)+\tr(\xb\yb)\tr(\zb\wb)\\ \nonumber
&-&\tr(\xb)\tr(\yb)\tr(\zb\wb)-\tr(\xb)\tr(\wb)\tr(\yb\zb)\\ 
&-&\tr(\yb)\tr(\zb)\tr(\xb\wb)-\tr(\zb)\tr(\wb)\tr(\xb\yb))\label{eqn:4relation}
\end{eqnarray}  

Therefore, words of length 4 or longer are not needed to generate the ring $\C[\X_r]$.    

Let $\mathcal{G}(r)$ be the union of the sets
$\mathcal{G}_1=\{\tr(\xb_1),...,\tr(\xb_r)\}$ of cardinality $r$,
$\mathcal{G}_2=\{\tr(\xb_i\xb_j)\ |\ 1\leq i,j \leq r\ \text{and}\ i\not=j\}$ of cardinality $\frac{r(r-1)}{2}$, and 
$\mathcal{G}_3=\{\tr(\xb_i\xb_j\xb_k)\ | \ 1\leq i<j<k \leq r \}$ of cardinality $\frac{r(r-1)(r-2)}{3}$.  
  
Using Equation \eqref{eqn:4relation} one concludes:  
  
\begin{theorem}[\cite{Vogt},  \cite{PrSi}]
$\mathcal{G}(r)$ is a minimal generating set for $\C[\X_r]$.
\end{theorem}
  
Since the dimension of $\X_r$ is $3r-3$ for $r\geq 2$ and 1 if $r=1$ we conclude that $\X_1=\C$ and $\X_2=\C^3$.  
   
In \cite{G9} it is shown that $\X_3$ is the spectrum of the following ring:  $$\C[t_1,t_2,t_3,t_{12},t_{13},t_{23},t_{123}]/\langle t_{123}^2-Pt_{123}+Q\rangle$$ where $P=t_1t_{23}+t_2t_{13}+t_3t_{12}-t_1t_2t_3$ and $$Q=t_1^2+t_2^2+t_3^2 +t_{12}^2+t_{23}^2+t_{13}^2-t_1t_2t_{12}-t_2t_3t_{23}-t_1t_3t_{13}+t_{12}t_{23}t_{13}-4.$$

The isomorphism is determined by the surjection $$\C[t_1,t_2,t_3,t_{12},t_{13},t_{23},t_{123}]\to \C[\X_3]$$ given by $t_1\mapsto \tr(\xb)$, $t_2\mapsto \tr(\yb)$, $t_3\mapsto \tr(\zb)$, $t_{12}\mapsto \tr(\xb\yb)$, $t_{13}\mapsto \tr(\xb\zb)$, $t_{23}\mapsto \tr(\yb\zb)$, and $t_{123}\mapsto \tr(\xb\yb\zb)$.

Although the moduli space $\X_r$ embeds in $\C^{N}$ where $N\geq\frac{r(r^2+5)}{6}$, its dimension is only $3r-3$ for $r\geq 2$.   So having described a minimal global coordinate system on $\X_r$, we now turn our attention to local coordinates; that is, to maximal algebraically independent subsets of $\mathcal{G}(r)=\mathcal{G}_1\cup\mathcal{G}_2\cup\mathcal{G}_3$.

\begin{theorem}
The following subsets of $\mathcal{G}(r)$ are together algebraically independent:
$\{\tr(\xb_i)\ | \ 1\leq i \leq r\}$, $\{\tr(\xb_1\xb_i)\ |\ 2\leq i\leq r\}$, and $\{\tr(\xb_2\xb_i)\ |\ 3\leq i\leq r\}$. There are $r+(r-1)+(r-2)=3r-3$ of these generators, and so the union of these three sets forms a maximal set of algebraically independent generators.
\end{theorem}
  
\begin{proof}
This is an adaptation of the proof Aslaksen, Tan, and Zhu used in 1994 to establish a similar result for $\C[\Y_r]$, see \cite{ATZ}.  Teranishi in 1988 gave a different proof in \cite{T3}.  

We prove this by induction.  For $r=1,2,3$ this has already been established.  
For $r\geq 4$ we calculate the Jacobian matrix of these $3r-3$ functions in $3r-3$ independent variables:
from $\xb_1$ we take $x_{11}^1$, from $\xb_2$ we take $x^2_{11},x^2_{22}$, and from $\xb_k$ we take $x^k_{11},x^k_{12},x^k_{22}.$  Putting $\tr(\xb_r),\tr(\xb_1\xb_r), \tr(\xb_2\xb_r)$ in the last $3$ rows of the Jacobian we get a block diagonal matrix. By induction we must show these three traces are independent in the variables from $\xb_r$. We first show that the matrix entries we have chosen are independent:  
  
Generically, we can assume the first matrix is diagonal, so $\xb_1=\left(
 \begin{array}{cc}
  x_{11}^1& 0 \\
  0& 1/x^1_{11}  \\
  \end{array}\right).$ Then conjugating by $\left(
 \begin{array}{cc}
 1/\sqrt{x^2_{21}}& 0 \\
 0& \sqrt{x^2_{21}}  \\
 \end{array}\right)$ allows us to assume that $\xb_2=\left(
 \begin{array}{cc}
  x^2_{11}& x^2_{11}x^2_{22}-1 \\
  1& x^2_{22}  \\
  \end{array}\right).$
   
Lastly, using $\mathrm{det}(\xb_k)=1$ allows us to solve for $x^k_{21}$ in the other generic matrices.  Since this leaves us with only $3r-3$ elements, they must be independent since the dimension is $3r-3$ also.
  
Using {\it Mathematica}, we calculate this subdeterminant of the Jacobian and evaluate at random unimodular matrices; finding it non-zero.  If there was a relation the determinant would be identically zero and so any non-zero evaluation shows independence.
\end{proof}

\begin{remark}
It is natural to ask how much can be recovered from these parameters.  In the cases, $r=1,2,3$ the natural map $\X_r\to \C^{3r-3}$ is surjective and there is always a slice $($a map back that commutes$)$.   Unfortunately, this is not always the case.
By {\it \cite{Fl},} $\X_r\longrightarrow \C^{3r-3}$ is only surjective in the cases $r=1,2,3$; but in general the image omits only a subset of a codimension 1 subspace.
\end{remark}

\begin{remark}
Minimal generators and maximal algebraically independent subsets of those generators for the coordinate ring of $\X(F_r,\SL_3(\C))$ were described in \cite{La0,La1,La3,La4}, and for $\X(F_2,\SL_4(\C))$ in \cite{GL}.  The problem is open for $\X(F_r,\SL_4(\C))$ when $r\geq 3$ and $\X(F_r,\SL_n(\C))$ for $n\geq 5$ and $r\geq 2$. 
\end{remark}

\subsection{Finitely Presented Groups}

For any complex reductive affine algebraic group $G$, if $\Gamma$ is a group on $r$ generators, then the epimorphism $F_r\to \Gamma$ mapping the free generators of $F_r$ to the generators of $\Gamma$ gives an embedding $\hm(\Gamma, G)\hookrightarrow \hm(F_r,G)$.  This in turn gives a ring epimorphism on coordinate rings $\C[\hm(F_r,G)]\to \C[\hm(\Gamma, G)]$, and also on invariant rings $\C[\hm(F_r,G)]^G\to \C[\hm(\Gamma, G)]^G$.  We thus obtain an embedding $\X(\Gamma,G)\hookrightarrow \X(F_r,G)$ depending on the original epimorphism $F_r\to \Gamma$. Therefore, any generating set for $\C[\X(F_r,G)]$ gives a generating set for $\C[\X(\Gamma, G)]$.

From this observation, and the results of the previous subsection, we immediately conclude the following theorem.

\begin{theorem}
Let $\Gamma$ be a finitely generated group. Then there exists $r$ so that $\C[\X(\Gamma, \SL_2(\C))]$ is generated by $\mathcal{G}(r)$.  
\end{theorem}

As shown in \cite{LaSi}, although there exist complex reductive groups $G$ where trace functions do not generate $\C[\X(F_r,G)]$, the regular functions on
such character varieties are always rational functions in characters.  

Even when $G=\SL_2(\C)$, when $\Gamma$ is not free (like when it is free abelian of rank $r$, for example) the generating set $\mathcal{G}(r)$ need {\it not} be minimal (or even unshortenable); see 
\cite{SiGenerating}.  However, standard Gr\"obner basis elimination algorithms can reduce the collection $\mathcal{G}(r)$ to an unshortenable generating set given the defining relations of $\X(\Gamma, G)$. In Section \ref{section-relations} we determine explicitly the defining relations of $\X(\Gamma, G)$ that cut out the variety as a set for any $\Gamma$ (when $\Gamma$ is a free group these relations cut out the variety as a scheme).

\subsection{Rank 1 Groups}

Every simple rank 1 complex algebraic group $G$ is isomorphic to $\SL_2$ or $\PSL_2$ since the simply connected group with Lie algebra $\mathfrak{sl}_2$ is unique and clearly $\SL_2$ and all others are central quotients thereof.  As mentioned in the previous subsection, to determine a generating set for $\C[\X(\Gamma, G)]$ for any finitely generated $\Gamma$ it suffices to find a generating set for $\C[\X(F_r,G)]$ where $F_r\to \Gamma$ is a surjection.

Now, $\PSL_2\cong \SL_2/Z$ where $Z=\{\pm \id\}$ is the center.  Since the $G$-conjugation action on $\hm(F_r, G)$ commutes with the left multiplication action of $Z^r$ on $\hm(F_r,G)\cong G^r$ given by $(z_1,...,z_r)\cdot (g_1,...,g_r)=(z_1g_1,...,z_rg_r)$, we conclude that $\X(F_r,\PSL_2(\C))\cong\X(F_r,\SL_2(\C))\quot Z^r$.  

Sikora describes in \cite[Section 3]{SiFinite} how to start from a generating set for $\C[\X(F_r,\SL_2(\C))]$ and determine one for $$\C[\X(F_r,\SL_2(\C))]^{Z^r}=\C[\X(F_r,\PSL_2(\C))].$$  We summarize here.

Let $F_r=\langle \gamma_1,...,\gamma_r\rangle.$ The abelianization map $$C:F_r\to F_r/[F_r,F_r]\cong \Z^r$$ allows one to define the degree of a word $\gamma_i$ in $\gamma$ since $C(\gamma)=\gamma_1^{d_1}\cdots\gamma_r^{d_r}$; define $\deg_{i}(\gamma)=d_i$.  Now let $v:F_r\to(\Z/2\Z)^r$ be defined by $v(\gamma)=(\deg_1(\gamma)\mod 2,...,\deg_r(\gamma)\mod 2)$. Let $\varphi:\X(F_r,\SL_2(\C))\to \X(F_r,\PSL_2(\C))$ be the quotient map induced by $\SL_2\to \PSL_2$.

Given the isomorphism $\X(F_r,\SL_2(\C))\quot Z^r\cong \X(F_r,\PSL_2(\C))$ induced by $\varphi$, for every $g_1,...,g_n\in F_r$ such that $\deg(g_1,...,g_n):=\sum_{i=1}^nv(g_i)=0$ there exists a unique $\lambda_{g_1,...,g_n}\in \C[\X(F_r,\PSL_2(\C))]$ such that $$\tr(g_1)\cdots \tr(g_n)=\lambda_{g_1,...,g_n}\circ \varphi.$$  Now let $$\overline{\mathcal{G}}(r)=\{\lambda_{g_1,...,g_n}\ |\ g_1,...,g_n\in \mathcal{G}(r)\text{ and }\deg(g_1,...,g_n)=0\}.$$  Then \cite[Theorem 7]{SiFinite} shows $\C[\X(F_r,\PSL_2(\C))]$ is generated by $\overline{\mathcal{G}}(r)$.

As an example (see \cite[Corollary 10]{SiFinite}), since $\C[\X(F_2,\SL_2(\C))]=\C[\tr(\xb),\tr(\yb),\tr(\xb\yb)]$ we conclude $$\C[\X(F_2,\SL_2(\C))]=\C[\tr(\xb)^2,\tr(\yb)^2,\tr(\xb\yb)^2,\tr(\xb)\tr(\yb)\tr(\xb\yb)].$$

\section{Relations of the Character Variety}\label{section-relations}

Finding presentations for coordinate rings $\C[\X(\Gamma, \SL_2(\C))]$ of $\SL_2(\C)$-character varieties is a ``cottage industry."  In part, this stems from $\PSL_2(\C)$ being the isometry group for the upper-half space model of hyperbolic 3-space.  Knot complements are 3-manifolds, most of which admit hyperbolic structures, and Mostow rigidity implies such hyperbolic structures are unique (up to conjugation).  Therefore, the $\PSL_2(\C)$-character variety of $\Gamma$, when $\Gamma$ is the fundamental group of the manifold, is naturally related to knot invariants and hyperbolic geometry (see \cite{CS} for the beginning of the story).  For example, \cite{PrSi} shows that the ``classical limit'' of the Kauffman bracket skein module of the 3-manifold is the coordinate ring of the corresponding character variety.  In \cite{CCGLS} a knot invariant called the A-polynomial is defined in terms of character varieties, and in \cite{Petersen-A} it is shown how to compute certain A-polynomials using the explicit structure of the coordinate ring.  

Here are some example recent papers concerning the structure of $\SL_2(\C)$ or $\PSL_2(\C)$-character varieties for various cases of $\Gamma=\pi_1(M)$: \cite{MPL} describes the case of two-bridge knots, \cite{BaPe} describes the case of once-punctured torus bundles, \cite{TranPet} handles double twist links, \cite{Tran-pretzel} concerns the case of pretzel knots, \cite{Oller, MaOl, Munoz-torus} completely describe the case of torus knots and links, \cite{HiLo} handles periodic knots and links, \cite{Landes} concerns the Whitehead link and certain hyperbolic 2-bridge links, and \cite{Qa} and \cite{Tran-onerelator} describe certain general families of 1-relator groups $\Gamma$ in the same spirit.

There are other applications of this structure. For example, in \cite{G9} Goldman uses the structure of the coordinate ring to describe the Fricke spaces of hyperbolic structures on surfaces, and in \cite{GX, GoXi-Torelli} Goldman and Xia use the explicit structure of the coordinate ring to analyze dynamics on certain moduli spaces of interest (called {\it relative} character varieties). 

There are many other examples where the explicit structure of the coordinate ring of $\X(\Gamma, \SL_2)$ is important, and we apologize for not citing everyone.

In this section we describe a fast and effective algorithm that takes as input any finite presentation of a discrete group $\Gamma$ and produces as output an explicit presentation of $\C[\X(\Gamma, \SL_2(\C))]/\sqrt{0}$.  When $\Gamma$ is not a free group, as shown in \cite{KaMi, KaMi2} and more generally in \cite{Rap} the character variety may not be reduced and so we only obtain the set-theoretic cut-out of the variety $\X(\Gamma, \SL_2)$.  However, when $\Gamma$ is free or when it is a closed surface (among some other examples) it is known $\X(\Gamma, \SL_2)$ is reduced (see \cite[Theorem 7.3]{PrSi}) and so we obtain the scheme-theoretic cut-out since our description does not have nilpotents.

\subsection{Free Groups}\label{sec:fgrels}

In 2003, Drensky \cite{D} determined the ideal of relations for $\C[\Y_r]$.  We now write the relations for $\C[\X_r]$ from those in $\C[\Y_r]$.  

\begin{theorem}\label{thm-scheme}
The scheme $\mathcal{X}_r=\hm(F_r,\SL_2(\C))\quot \SL_2(\C)$ is isomorphic to $$\mathrm{Spec}\left(\C[t_1,...,t_{\frac{r(r^2+5)}{6}}]/\mathfrak{I}_r\right)$$ where $\mathfrak{I}_r$ is an ideal generated by $\frac{1}{2}\left(\binom{r}{3}^2+\binom{r}{3}\right)+r\binom{r}{4}$ polynomials described as follows.

Let $\zb_i=\xb_i-\frac{1}{2}\tr(\xb_i)\id$ be generic traceless matrices, and let $$s_3(\ab_1,\ab_2,\ab_3)=\sum_{\sigma\in S_3}\mathrm{sign}(\sigma) \ab_{\sigma(1)}\ab_{\sigma(2)}\ab_{\sigma(3)}.$$
 
Type 1 relations:  $$\tr(s_3(\zb_{i_1},\zb_{i_2},\zb_{i_3}))\tr(s_3(\zb_{j_1},\zb_{j_2},\zb_{j_3}))+18\det(\tr\left((\zb_{i_{row}}\zb_{j_{column}})\right)=0,$$ for $1\leq i_1<i_2<i_3\leq r$, $1\leq j_1<j_2<j_3\leq r$.   
   
Type 2 relations: $$\sum_{k=0}^3(-1)^k\tr(\zb_i\zb_{p_k})\tr(s_3(\zb_{p_0},...,\widehat{\zb_{p_k}},...,\zb_{p_3}))=0,$$ where $1\leq i\leq r$, $1\leq p_0<p_1<p_2<p_3\leq r$ and $\widehat{\zb}$ means omission.  
\end{theorem}

\begin{proof} 
Since $\mathcal{X}_r$ is reduced and irreducible, being the affine GIT quotient of a smooth group scheme, \cite[Proposition 2.6]{Ha} implies the scheme structure of $\mathcal{X}_r$ is determined by the structure of the variety $\X_r$. 

Recall, $$\C[\Y_r]/\Delta_r=\C[x^k_{ij}]^{\SL_2(\C)}/\Delta_r\approx \left( \C[x^k_{ij}]/\Delta_r \right)^{\SL_2(\C)}=\C[\X_r],$$  where $$\Delta_r=\langle x^{k}_{11}x^{k}_{22}-x^{k}_{12}x^{k}_{21}-1\ |\ 1\leq k\leq r\rangle,$$ namely, the ideal generated by determinants. 

The minimal generating set for $\C[\X_r]$ described in the previous section has cardinality $\binom{r}{3}+\binom{r}{2}+\binom{r}{1}=\frac{r(r^2+5)}{6}$.  These same generators lift to generators of $\C[\Y_r]$ when the $r$ determinants are included in the generating set (and give a minimal generating set for $\C[\Y_r]$).  Let $N_r=\binom{r}{3}+\binom{r}{2}+2\binom{r}{1}$.  Denote this minimal generating set for $\C[\Y_r]$ by $t_1,...,t_{N_r}$, and let $\mathfrak{J}$ be the ideal of relations for $\C[\Y_r]$ in those generators.   Then $\C[\Y_r]=\C[t_1,...,t_{N_r}]/\mathfrak{J}$.  Since the quotient of $\C[\Y_r]$ by $\Delta_r$ exactly specializes the $r$ determinant generators to 1 (which are algebraically independent parameters), the ideal of relations for $\X_r$ is exactly the corresponding specialization of $\mathfrak{J}$.  So, it suffices to solve the problem for $\C[\Y_r]$.  We note that since $\xb^2-\tr(\xb)\xb+\det(\xb)\id=\mathbf{0}$, specializing the determinants of the $r$ generic matrices to 1 is the same as setting $\tr(\xb_i^2)=\tr(\xb_i)^2-2$ in the generators of $\mathfrak{J}$.

We sketch the treatment in \cite{DF} for the description of $\C[\Y_r]$.  The key idea is that the conjugation action of $\PSL_2(\C)$ on $\mathfrak{sl}_2(\C)$ is equivalent to the orthogonal action of $\mathsf{SO}_3(\C)$ on $\mathfrak{so}_3(\C)$ where the killing form in the former takes the place of the dot product in the latter.

Indeed, \begin{align*}\mathfrak{gl}_2(\C)^{\times r}\quot \SL_2(\C)&=\mathfrak{gl}_2(\C)^{\times r}\quot \mathsf{PSL}_2(\C)\\&=\mathfrak{gl}_2(\C)^{\times r}\quot \mathsf{SO}_3(\C)\\&\cong \C\left(\frac{x_{11}+x_{22}}{2}\right)^{\times r}\bigoplus \mathfrak{so}_3(\C)^{\times r}\quot \mathsf{SO}_3(\C),\end{align*} where the coordinates for $\mathfrak{gl}_2(\C)$ are $\{x_{11},x_{21},x_{12},x_{22}\}$.

However, Weyl solved the relations problem for orthogonal invariants in 1939 in his book {\it The Classical Groups, Their Invariants and Representations} \cite{Weyl}. Rewriting those invariants in terms of traces then gives the Type 1 and Type 2 relations stated in the above theorem (please see \cite{DF} or \cite{D} for details).  Since the quotient by $\mathfrak{J}$ is reduced, and $\mathfrak{J}$ cuts out an irreducible variety, the result follows.
\end{proof}

\begin{remark}
In \cite{GoMo}, using results of \cite{Magnus, Vogt}, a different, but necessarily equivalent, description of $\X_r$ is given by their Theorem 3.1.  In their description, there are four families of relations $($see their Section 5$)$ as opposed to our description which consists of only two families of relations.  There is no contradiction since there are many different bases for the same ideal, and the cardinality of different bases for an ideal need not be the same.
\end{remark}

\begin{remark}
In \cite{La0, La1} the relations problem was solved for $\X(F_2,\SL_3)$.  For higher rank free groups $F_r$ and $G=\SL_3$ the relations problem remains open.  There is little to nothing known about the relations for $\X(F_r,\SL_n)$ for $r\geq 2$ and $n\geq 4$.  
\end{remark}

\subsection{Finitely Presented Groups}\label{sec:fingrels}

For a finitely presented group $\Gamma$, $\X(\Gamma,G)$ is always cut out of $\X(F_r,G)$ by using the relations in $\Gamma$. This can be made explicit when $G=\SL_2(\C)$ as per results in \cite{GoMo}, \cite{Magnus}, and \cite{Vogt}.
 
\begin{theorem}[Theorem 3.2, \cite{GoMo}]\label{gomo-thm}
Let $\Gamma=\langle \gamma_1,....,\gamma_r\ |\  R_i,\ i\in I\rangle$, and denote $\gamma_0=1$.  Then the closed points of $\X(\Gamma,\SL_2(\C))$ are given by $$\{[\rho]\in \X(F_r,\SL_2(\C))\ |\ \mathrm{tr}(\rho(R_i\gamma_j))-\mathrm{tr}(\rho(\gamma_j))=0,\forall i,j\}.$$ 
\end{theorem}

\begin{remark}
The relations $\mathrm{tr}(\rho(R_i\gamma_j))-\mathrm{tr}(\rho(\gamma_j))=0$ above cutting $\X(\Gamma, \SL_2)$ out of $\X(F_r,\SL_2)$ do not generally give the scheme structure of $\X(\Gamma,\SL_2)$.  It gives only a set-theoretic cut-out, i.e., an algebraic bijection to the maximal ideals in $\C[\X(\Gamma,\SL_2)]$ and so determines the structure of $\C[\X(\Gamma, \SL_2)]/\sqrt{0}$.
\end{remark}

\begin{remark}
It is an open problem to determine analogues of Theorem \ref{gomo-thm} for $\SL_n(\C)$ for $n\geq 3$.  This would be very useful.
\end{remark}
 
We now sketch the proof in \cite{GoMo} for the sake of the reader.

\begin{proof}
For $$[\rho]\in \{[\rho]\in \X(F_r,\SL_2(\C))\ |\ \mathrm{tr}(\rho(R_i\gamma_j))-\mathrm{tr}(\rho(\gamma_j))=0,\forall i,j\},$$ either $\rho(R_i)=\id$ for all $i$ (in which case we are done) or $\rho(R_{i_0})$ is parabolic for some $i_0$.  In the latter case, up to conjugation, we can assume $\rho(R_{i_0})=\left(\begin{array}{cc}1&1\\0&1\end{array}\right)$. But then, since $\tr(\rho(R_{i_0})\rho(\gamma_i))=\tr(\rho(\gamma_i))$ for all $i$, we conclude that $\rho$ is upper-triangular.  Thus, the semisimple representative $\rho^{ss}\in [\rho]$ is diagonal.  Then $\tr(\rho^{ss}(R_{i_0}))=\tr(\id)=2$ implies $\rho^{ss}(R_{i_0})=\id$.  This completes the proof, since the same conclusion follows for any $i$ where $\rho(R_i)$ is parabolic, and the conjugation orbit of the semisimple representative in $[\rho]$ is the unique closed sub-orbit.
\end{proof}

\subsection{The Algorithm}
The trace identities described in Section \ref{sec:freegroups} can be implemented into a \emph{trace reduction} algorithm. The symbolic computation software \emph{Wolfram Mathematica} \cite{mathematica} provides the structure of replacement rules and patterns, which makes this task simple.  We also have made available a version in {\it SageMath} and Python for those who do not have {\it Mathematica}.

When given an algebraic expression, for example a polynomial, \emph{Mathematica} will automatically collect terms and reorder them alphabetically. We can define a new type of data for traces of words called \emph{Tr}, and replacement rules for this data type which {\it Mathematica} will follow until the trace is completely simplified.

We define the following rules :
\begin{enumerate}
   \item The trace of the empty word is replaced by the value $2$, the trace of the identity matrix in $\SL_2(\C)$.
   \item The trace of a word containing a negative power of any letter is simplified using the identity
   \[\tr(\ub\xb^{-k}\vb) = \tr(\xb^k)\tr(\vb\ub)-\tr(\ub\xb^{k}\vb).\]
   \item The trace of any word containing a power of a letter larger than or equal to two is simplified using the identity
   \[\tr(\ub\xb^2\vb)=\tr(\xb)\tr(\xb\vb\ub)-\tr(\vb\ub).\]
   \item The trace of any word is replaced by the cyclic permutation of its letters which is minimal for the lexicographic order.
 \item Any word of length greater than three is reduced using Equation \eqref{eqn:4relation}.
\end{enumerate}

 Every time one of the above cases happens, \emph{Mathematica} will automatically replace the expression with its simplified version, recursively until none of the rules apply. For instance,
 \begin{example}
 \begin{verbatim}

 In[179]:= Tr[Word[1,2,-3,1]]
 Out[179]= Tr[Word[3]](-Tr[Word[2]]+Tr[Word[1]]
 Tr[Word[1,2]])+Tr[Word[2,3]]-Tr[Word[1]]Tr[Word[1,2,3]]
 \end{verbatim}
 \end{example}
 
 The expression $\mathrm{Tr[Word[}i_1,\dots,i_k\mathrm{]]}$ represents the trace of the word $\xb_{i_1}\xb_{i_2}\cdots \xb_{i_k}$, and by convention $\xb_{-i}=\xb_i^{-1}$ for $i>0$.
 Here we have reduced the trace of the word $\xb_1\xb_2\xb_3^{-1}\xb_1$ to an expression using only traces of words of length three or less. The function \emph{ToVariables} makes the result more easily readable.
 \begin{example}
   \begin{verbatim}

     In[180]:=ToVariables[Tr[Word[1,2,-3,1]]]
     Out[180]=
   \end{verbatim}
   \begin{center}$t_{\{3\}} \left(t_{\{1\}} t_{\{1,2\}}-t_{\{2\}}\right)+t_{\{2,3\}}-t_{\{1\}} t_{\{1,2,3\}}$
   \end{center}
\end{example}
$t_{\{i,j,k\}}$ represents the trace of the word $\xb_i\xb_j\xb_k$.
 Once the trace reduction algorithm is in place, finding a presentation for the $\SL_2(\C)$-character variety of any finitely presented group is simple. There are two types of relations to compute, the free group relations and the relations coming from the presentation of the group.

 The free group relations subdivide into ``Type 1'' and ``Type 2'' (Section \ref{sec:fgrels}). The functions \emph{rel1} and \emph{rel2} in the notebook respectively output each type of relation given the right number of parameters. The relations are in terms of traces so are automatically simplified by the substitution rules described above.

  The function \emph{presrel} takes a word and a number of generators. It outputs a collection of relations for the coordinate ring of the character variety of a group having this word in its presentation (Section \ref{sec:fingrels}). For a relation given by the word $\ub$, the corresponding character variety relations are
  \[\tr(\ub\xb_i) - \tr(\xb_i)=0.\]

  Finally, the function \emph{Presentation} makes use of the previously described functions in order to take as input a finite presentation of a group $\Gamma$ and outputs a presentation for the (reduced) coordinate ring of the $\SL_2(\C)$-character variety.

The input to \emph{Presentation} takes the form of an integer, the number of generators, followed by a list of words, the relations in the presentation of $\Gamma$. The output is a pair, the generators and the relations defining $\X(\Gamma,\SL_2(\C))$.
  \begin{example}
    In this example, we use the presentation $$\Gamma=\{a,b~|~abab\}.$$
    \begin{verbatim}
In[216]:= Presentation[2,{Word[1,2,1,2]}]
Out[216]=
    \end{verbatim}
\begin{align*}
\Big\{ 
\{ t_{\{1\}},& t_{\{2\}}, t_{\{1,2\}} \},\\ 
&\left\{
t_{\{1,2\}}^2-4 , t_{\{1\}} t_{\{1,2\}}^2-t_{\{2\}}
   t_{\{1,2\}}-2 t_{\{1\}} , t_{\{2\}}
   t_{\{1,2\}}^2-t_{\{1\}} t_{\{1,2\}}-2 t_{\{2\}} 
   \right\}\Big\}
\end{align*}     
\end{example}

The program described above is available at \url{http://math.gmu.edu/~slawton3/trace-identites.nb}.See also \url{http://math.gmu.edu/~slawton3/Main.sagews}, and \url{http://math.gmu.edu/~slawton3/charvars.py}.

\section{Examples}
\subsection{Free Groups}\label{rank4subsection}

In Section \ref{sec:freegroups} we described $\X_r$ for $r=1,2,3$.  As a first demonstration of the above algorithm, we now handle the $r=4$ case.

The fundamental group of the 5-holed sphere is a free group on four letters with the following presentation: $$\pi=\langle a, b, c, d, e\ |\ abcde=1\rangle\cong \langle a,b,c,d\rangle,$$ since $e^{-1}=abcd$.   The character variety is by definition $$\hm(\pi,\SL_2(\C))\quot \SL_2(\C)\cong \SL_2(\C)^{\times 4}\quot \SL_2(\C),$$ given by $$[\rho]\mapsto [(\rho(a),\rho(b),\rho(c),\rho(d))]=[(\ab,\bb,\cb,\db)].$$ So this is the moduli space of (polystable) flat $\SL_2(\C)$-bundles over the 5-holed sphere.

The coordinate ring has the following presentation:
$$\C[\SL_2(\C)^{\times 4}\quot \SL_2(\C)]=\C[r_1,...,r_9][t_1,...,t_5]/(f_1,...,f_{14}),$$ where $\{r_1,...,r_9\}$ is a minimal generating set for the rational function field, $\{r_1,...,r_9,t_1,...,t_5\}$ is a minimal generating set for the coordinate ring, and $\{f_1,...,f_{14}\}$ is a minimal generating set for the ideal of relations in terms of the generators. 
  
We obtain the following consequences. $\SL_2(\C)^{\times 4}\quot \SL_2(\C)$ minimally embeds in $\C^{14}$ and its dimension is $9$.  Since it has 14 relations it is very far from a complete intersection like the rank $1,2,3$ cases.   More still, at a generic smooth point $[\rho]$, $\{dr_1,...,dr_9\}$ generates $$T^*_{[\rho]}\left(\SL_2(\C)^{\times 4}\quot\SL_2(\C)\right)\cong \C^9.$$
 
Here are the formulas for the generators: 

$r_1=\tr(\ab),r_2=\tr(\bb),r_3=\tr(\cb), r_4=\tr(\db),r_5=\tr(\ab\bb),r_6=\tr(\ab\cb),r_7=\tr(\ab\db),r_8=\tr(\bb\cb),r_9=\tr(\bb\db)$
$t_1=\tr(\cb\db),t_2=\tr(\ab\bb\cb),t_3=\tr(\ab\bb\db),t_4=\tr(\ab\cb\db),t_5=\tr(\bb\cb\db)$

There are two types of relations (degree 5 and degree 6 in the matrix entries) and the rest is combinatorics.

In Appendix \ref{appendix-r4}, and also in the {\it Mathematica} notebook\footnote{Available at \url{http://math.gmu.edu/~slawton3/trace-identities.nb}}, we list the formulas for the ideal of relations for this case.  The polynomials $f_1$ through $f_4$ are all of one type (Type 2), $f_5$ through $f_8$ are the rank 3 relation for each set of 3, and $f_9$ through $f_{14}$ are a generalized relation of the same type as $f_5$ through $f_8$ (Type 1). 
   
\subsection{Finitely Presented Groups.}

Recall, $\X_2=\C^3$ and so for all $w\in F_2=\langle a,b\rangle$, there is a unique $P_w\in \C[x,y,z]$ such that $$\tr(w)=P_w(\tr(a),\tr(b),\tr(ab)).$$

The first example we consider is the figure eight knot (compare \cite[Proposition 3.1]{HP16}).   According to \cite{Rol76}, the fundamental group of the figure eight knot complement in $S^3$ admits a presentation $$\Gamma_{F8}=\Big\langle a, b\ |\ ab^{-1}a^{-1}bab^{-1}aba^{-1}b^{-1}\Big\rangle.$$  Let $w=ab^{-1}a^{-1}bab^{-1}aba^{-1}b^{-1}.$

The character variety $\X(\Gamma_{F8},\SL_2(\C))$ is given by $$\{\mathbf{v}\in \C^3\ |\ P_w(\mathbf{v})=2,\ P_{aw}(\mathbf{v})=x,\ P_{bw}(\mathbf{v})=y\}, $$ where $\mathbf{v}=(x,y,z)$.

Using the algorithm described in this paper, and simplifying with a Gr\"obner basis algorithm, we conclude that the $\SL_2(\C)$-character variety of the figure 8 knot is the zero locus of the ideal: $$\left\langle y^2-z-2,y^2 z-2 y^2-z^2+z+1,x-y\right\rangle.$$
Here is a picture of its $\R$-points:
\begin{center}
\includegraphics[scale=.4]{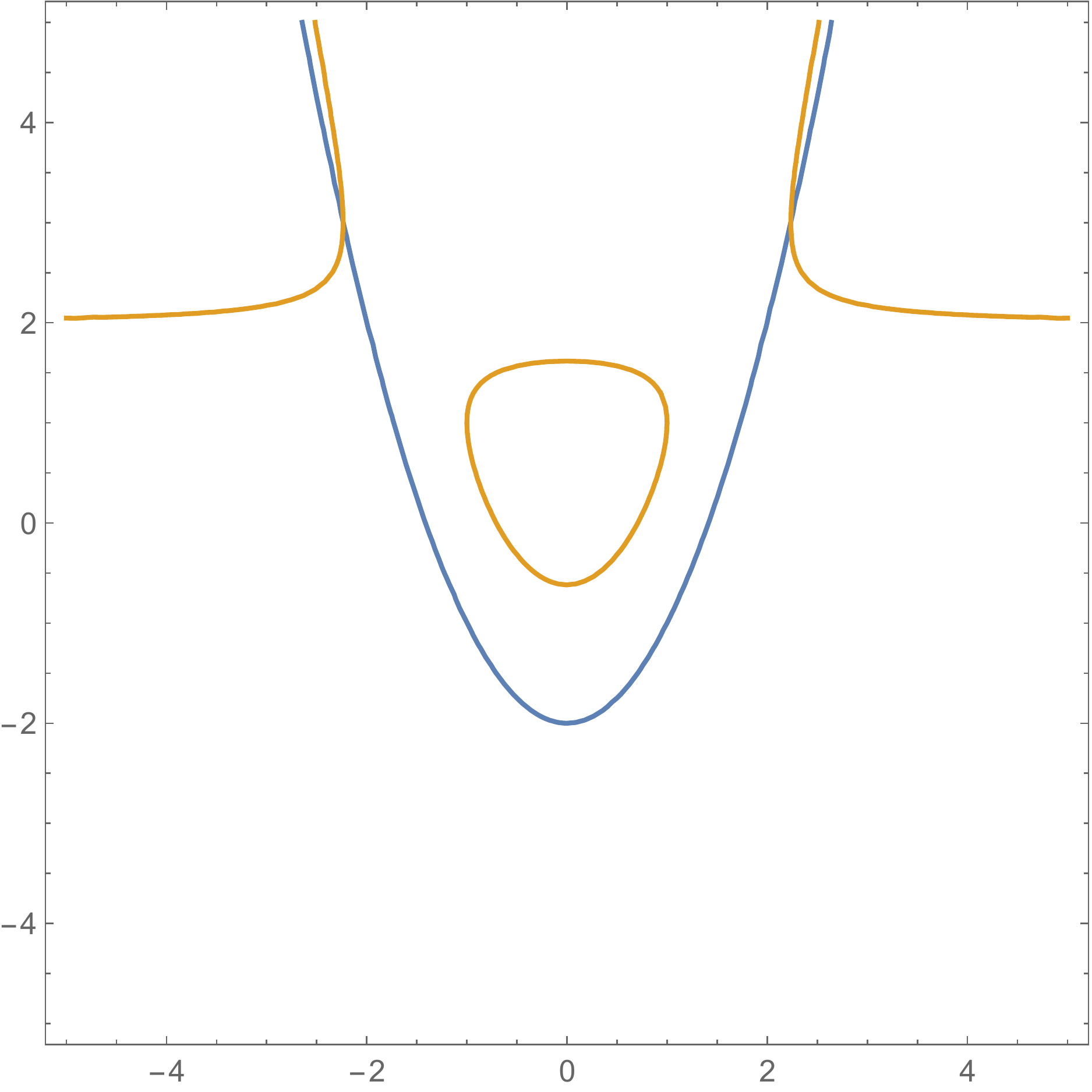}
\end{center} 
 
Clearly it has two components. With some algebraic manipulation (again using our program) one can show that this gives an equivalent solution as described in \cite[Proposition 3.1]{HP16}.  In particular, the component $y^2-z-2=0$ is the reducible representations while $y^2 z-2 y^2-z^2+z+1=0$ is the component of irreducible representations.
 
The next example we consider is the character variety of the Whitehead link complement (compare \cite{Landes}). 

The fundamental group of the Whitehead link complement in $S^3$ admits the presentation $$\Gamma_{WL}=\Big\langle a,b\ |\  abab^{-1}a^{-1}b^{-1}aba^{-1}b^{-1}a^{-1}baba^{-1}b^{-1}\Big\rangle.$$  Let $w=abab^{-1}a^{-1}b^{-1}aba^{-1}b^{-1}a^{-1}baba^{-1}b^{-1}$.
 
The character variety $\X(\Gamma_{WL},\SL_2(\C))$ is given by $$\{\mathbf{v}\in \C^3\ |\ P_w(\mathbf{v})=2,\ P_{aw}(\mathbf{v})=x,\ P_{bw}(\mathbf{v})=y\}. $$ 

Using the algorithm in this paper, we obtain a set of 3 polynomial relations.  Call them $P_1,P_2,$ and $P_3$.  Now $P_{abw}(\mathbf{v})=z$ is also a relation since $w$ evaluates to the identity at any representation of $\Gamma_{WL}$.  Define $P_4= P_{abw}(\mathbf{v})-z$.  The Gr\"obner basis of the ideal $\langle P_1,P_2,P_3,P_4\rangle$ is principally generated by $$P=\left(x^2+y^2+z^2-x y z-4\right) \left(-xy-2z+x^2z+y^2z-xyz^2+z^3\right).$$

Since $I_1=\langle P_1,P_2,P_3\rangle \subset \langle P_1,P_2,P_3,P_4\rangle=I_2$ we know that the corresponding algebraic sets satisfy $V(\langle P_1,P_2,P_3\rangle) \supset V(\langle P_1,P_2,P_3,P_4)\rangle$.  But $V(\langle P_1,P_2,P_3\rangle)$ cuts out the character variety by the algorithm, and as noted above every point in $V(\langle P_1,P_2,P_3)\rangle$ must also satisfy $P_4$.  Thus the algebraic sets are the same and so we conclude by the Nullstellensatz that $\sqrt{I_1}=\sqrt{I_2}$.  However, using polynomial division, one can check that $I_1\not=I_2$.

Regardless, the character variety of the Whitehead link complement corresponds to the hypersurface $\{\mathbf{v}\in \C^3\ |\  P=0\},$ which has two irreducible components.

From \cite{G9} we see the first component of this hypersurface is the trace of the commutator which corresponds exactly to the reducible representations.  The second component is exactly the equation defining the canonical component of the Whitehead link (see \cite[Proposition 4]{Landes}).  Moreover, as pointed out to us by the referee, using pairs of matrices found in \cite[Example 1]{wiel} that correspond to a discrete faithful representation (a hyperbolic structure) of the Whitehead link, one may check that $-xy-2z+x^2z+y^2z-xyz^2+z^3$ evaluates to zero (and so deduce it is a canonical component).  In particular, there is a discrete faithful representation of the Whitehead link determined by $\rho(a)=\left(\begin{array}{cc}1&i\\0&1\end{array}\right)$ and $\rho(b)=\left(\begin{array}{cc}1&0\\1+i&1\end{array}\right)$, and setting $x=\tr(\rho(a))=2$, $y=\tr(\rho(b))=2$ and $z=\tr(\rho(ab))=1+i$ reduces the defining equation of the component $-xy-2z+x^2z+y^2z-xyz^2+z^3$ to 0.

Here is a picture of the $\R$-points of the canonical component:
\begin{center}
\includegraphics[scale=.25]{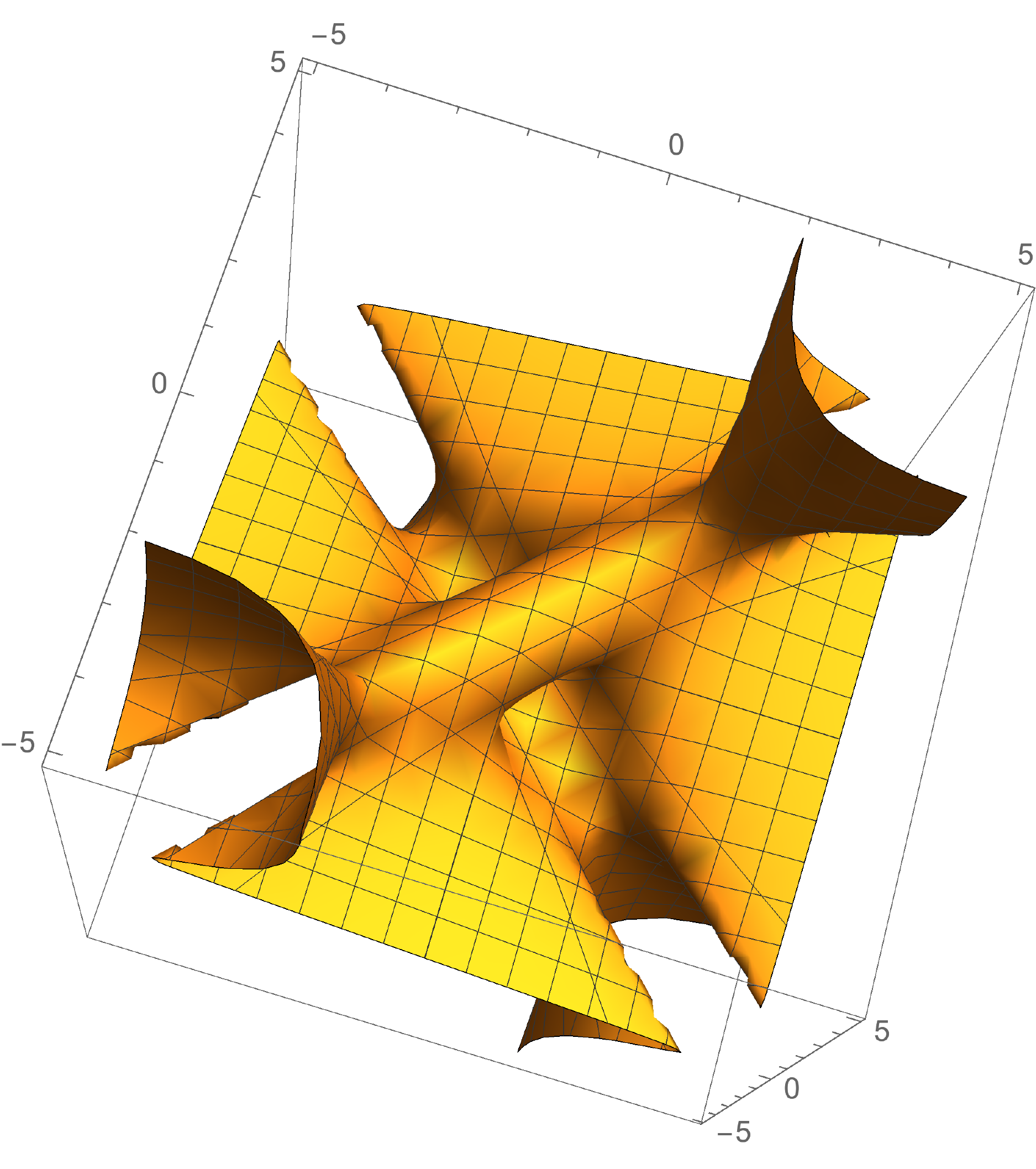}
\end{center}

Here is a picture of the $\R$-points of the component consisting of reducible representations:
\begin{center}
\includegraphics[scale=.25]{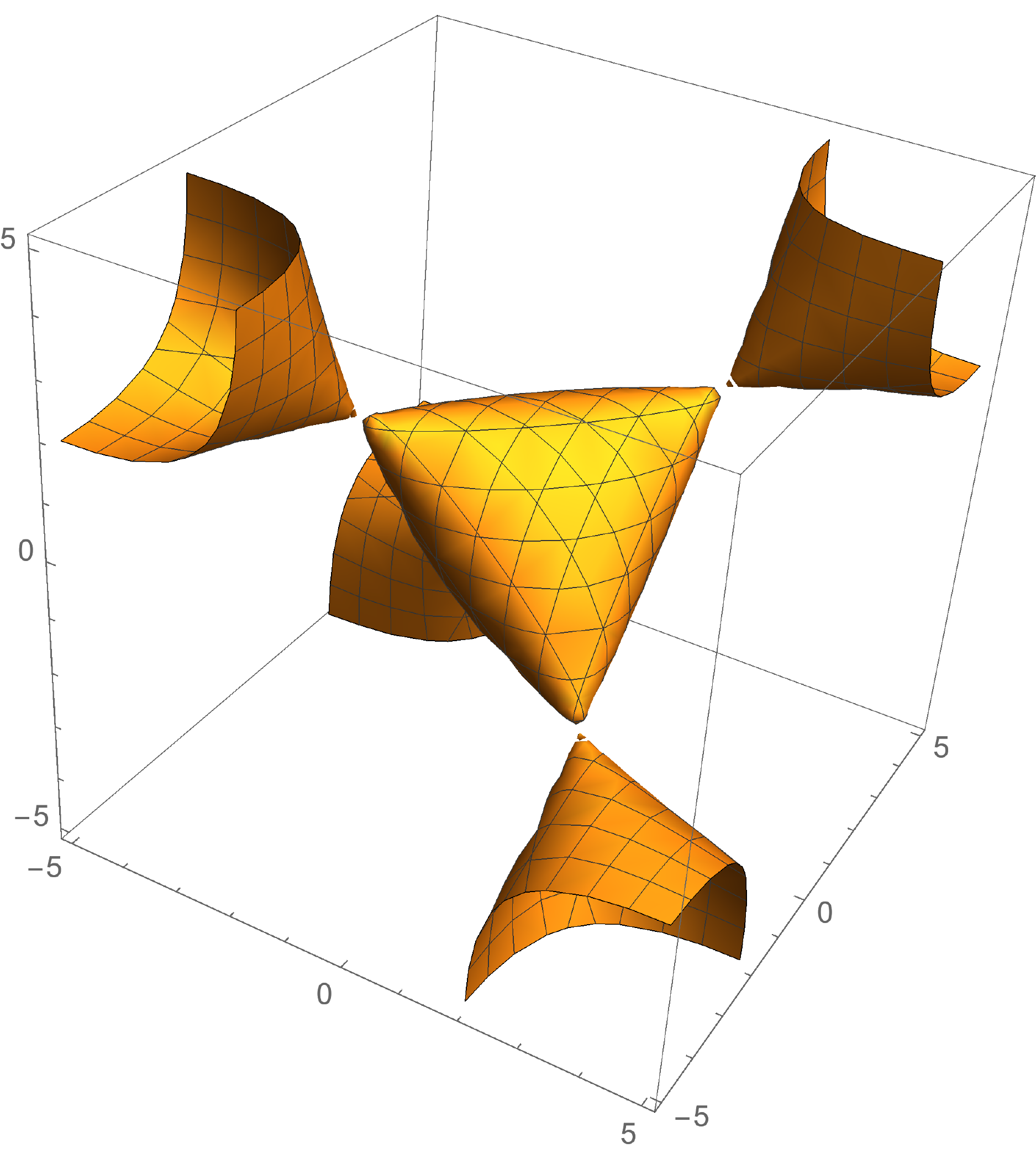}
\end{center}

To be more systematic, we can consider the (orientable closed) hyperbolic manifolds in the census in SnapPy \cite{SnapPy}, ordered by increasing hyperbolic volume.

The Weeks manifold (smallest volume hyperbolic manifold) is first in the list (see \cite{mom}).

By entering  \begin{verbatim}OrientableClosedCensus[0].fundamental_group()\end{verbatim} into SnapPy \cite{SnapPy} we obtain the following presentation for the Weeks manifold $\langle a,b\ |\ aabbaaBaB, aabbAbAbb\rangle$ where $A=a^{-1}$ and $B=b^{-1}$.

Entering this into our algorithm (and simplifying with a Gr\"obner basis algorithm) we obtain a presentation for its (reduced) character variety:

$\C[t_{\{1\}},t_{\{2\}},t_{\{1,2\}}]/
\langle t_{\{1,2\}}^6-3 t_{\{1,2\}}^5+2 t_{\{1,2\}}^4+4 t_{\{1,2\}}^3-12
   t_{\{1,2\}}^2+9 t_{\{1,2\}}-2,-t_{\{1,2\}}^5+3 t_{\{1,2\}}^4+t_{\{2\}}
   t_{\{1,2\}}^3-t_{\{1,2\}}^3-t_{\{2\}} t_{\{1,2\}}^2-6 t_{\{1,2\}}^2-3 t_{\{2\}}
   t_{\{1,2\}}+10 t_{\{1,2\}}+2 t_{\{2\}}-4,t_{\{1,2\}}^5-3 t_{\{1,2\}}^4+2
   t_{\{1,2\}}^3+5 t_{\{1,2\}}^2-13 t_{\{1,2\}}+t_{\{2\}}^3-t_{\{2\}}^2-3
   t_{\{2\}}+8,t_{\{1\}} t_{\{1,2\}}^2-t_{\{2\}} t_{\{1,2\}}^2+t_{\{1\}}
   t_{\{1,2\}}-t_{\{2\}} t_{\{1,2\}}-t_{\{1\}}+t_{\{2\}},-2 t_{\{1,2\}}^5+5
   t_{\{1,2\}}^4-t_{\{1,2\}}^3-9 t_{\{1,2\}}^2-t_{\{2\}}^2 t_{\{1,2\}}-t_{\{2\}}
   t_{\{1,2\}}+19 t_{\{1,2\}}+t_{\{1\}} t_{\{2\}}^2-t_{\{1\}}+t_{\{1\}}
   t_{\{2\}}-8,t_{\{1,2\}}^5-2 t_{\{1,2\}}^4+4 t_{\{1,2\}}^2-t_{\{1\}} t_{\{2\}}
   t_{\{1,2\}}-8 t_{\{1,2\}}+t_{\{1\}}^2+t_{\{2\}}^2\rangle.$

There are six relations that must be simultaneously satisfied.  Factoring the first we see $$\left(t_{\{1,2\}}-2\right) \left(t_{\{1,2\}}^2+t_{\{1,2\}}-1\right)\left( t_{\{1,2\}}^3-2 t_{\{1,2\}}^2+3 t_{\{1,2\}}-1\right)=0.$$ This says that $t_{\{1,2\}}$ has 6 allowed values.  Factoring the fourth relation gives $\left(t_{\{1\}}-t_{\{2\}}\right) \left(t_{\{1,2\}}^2+t_{\{1,2\}}-1\right)=0$ and so either $t_{\{1,2\}}=\frac{1}{2}\left(-1\pm\sqrt{5}\right)$, two of the allowed six values for this coordinate, or $t_{\{1\}}=t_{\{2\}}$.  In the first case where $t_{\{1,2\}}=\frac{1}{2}\left(-1\pm\sqrt{5}\right)$, the third equation specializes to $\left(t_{\{2\}}-2\right) \left(t_{\{2\}}^2+t_{\{2\}}-1\right)=0$ and so $t_{\{2\}}=2,$ or $\frac{1}{2}\left(-1\pm\sqrt{5}\right)$.  If $t_{\{2\}}=2,$ then the fifth equation then says that $t_{(1)}= \frac{1}{2}\left(-1\pm\sqrt{5}\right).$  If $t_{\{2\}}=\frac{1}{2}\left(-1\pm\sqrt{5}\right)$, then the sixth relation becomes $t_{\{1\}}^2+t_{\{1\}}-1$ and so $t_{\{1\}}=\frac{1}{2}\left(-1\pm\sqrt{5}\right)$ too.  Otherwise, $t_{\{1\}}=t_{\{2\}}$.  Now under that assumption, if $t_{\{1,2\}}=2$ then the relations imply $t_{\{1\}}=2,$ or $\frac{1}{2}\left(-1\pm\sqrt{5}\right)$.  Otherwise, the first relation determines the value of $t_{\{1,2\}}$ as a zero of the cubic  $t_{\{1,2\}}^3-2 t_{\{1,2\}}^2+3 t_{\{1,2\}}-1$ and the second relation implies $t_{\{1,2\}}^2-2 t_{\{1,2\}}-t_{\{1\}}+2=0$ which therefore determines $t_{\{1\}}$.  

We see the character variety is a {\it finite set}.  The intersection of this set with the polynomial $t_{\{1\}}^2+t_{\{2\}}^2+t_{\{1,2\}}^2+-t_{\{2\}} t_{\{1\}} t_{\{1,2\}}-4$ gives the reducible locus (non-empty since $(2,2,2)$ is such a point) and its complement gives the locus of irreducible representations (non-empty since the solutions to $t_{\{1,2\}}^3-2 t_{\{1,2\}}^2+3 t_{\{1,2\}}-1=0$ and $t_{\{1,2\}}^2-2 t_{\{1,2\}}-t_{\{1\}}+2=0$ and $t_{\{1\}}=t_{\{2\}}$ are such points).  This is consistent with the analysis in \cite[Example 4.1]{Harada}.

There are 11,031 manifolds in the census of closed orientable hyperbolic 3-manifolds (ordered by volume) in SnapPy \cite{SnapPy}.  

By entering \begin{verbatim}for M in OrientableClosedCensus[0:10]: print(M.fundamental_group()) \end{verbatim} SnapPy returns presentations for the first 10 manifolds in this census.

By entering exactly this output from SnapPy directly into our program using the command \begin{verbatim}FromSnapPyList["OUTPUT"]\end{verbatim} we obtain presentations of the first 10 character varieties on the census (please see the accompanying {\it Mathematica} notebook for this data).  This can be repeated for any of the manifolds in SnapPy.  \footnote{At the time of this writing, our Python program is expected to be incorporated directly into SnapPy.}

\appendix
\section{Ideal for the Rank 4 Case}\label{appendix-r4}

Without further ado, here are the defining relations for $\X_4$ (see Subsection \ref{rank4subsection}):
{\small
\begin{eqnarray*}
f_1&=&3 \tr(\bb\cb\db)\tr(\ab)^2-3 \tr(\cb\db)\tr(\bb)\tr(\ab)^2\\
   &-&3\tr(\bb\cb)\tr(\db)\tr(\ab)^2+3 \tr(\bb) \tr(\cb)\tr(\db)\tr(\ab)^2\\
   &+&3\tr(\ab\db)\tr(\bb\cb)\tr(\ab)-3 \tr(\ab\cb)\tr(\bb\db)\tr(\ab)\\
   &+&3\tr(\ab\bb) \tr(\cb\db)\tr(\ab)+3 \tr(\ab\cb\db)\tr(\bb)\tr(\ab)\\
   &-&3\tr(\ab\bb\db) \tr(\cb)\tr(\ab)-3\tr(\ab\db)\tr(\bb) \tr(\cb)\tr(\ab)\\
   &+&3\tr(\ab\bb\cb) \tr(\db)\tr(\ab)-3 \tr(\ab\bb)\tr(\cb) \tr(\db)\tr(\ab)\\
   &-&6\tr(\ab\db)\tr(\ab\bb\cb)+6\tr(\ab\cb)\tr(\ab\bb\db)-6\tr(\ab\bb)\tr(\ab\cb\db)\\
   &-&12\tr(\bb\cb\db)+6\tr(\cb\db) \tr(\bb)+6\tr(\ab\bb)\tr(\ab\db)\tr(\cb)\\
   &+&6 \tr(\bb\db)\tr(\cb)+6\tr(\bb\cb)\tr(\db)-6 \tr(\bb) \tr(\cb)\tr(\db),
\end{eqnarray*}

\begin{eqnarray*}
f_2&=&-3 \tr(\ab\cb\db)\tr(\bb)^2+3 \tr(\cb\db)\tr(\ab) \tr(\bb)^2\\
   &+&3\tr(\ab\db) \tr(\cb)\tr(\bb)^2-3 \tr(\ab) \tr(\cb)\tr(\db) \tr(\bb)^2\\
   &-&3\tr(\ab\db) \tr(\bb\cb)\tr(\bb)+3 \tr(\ab\cb)\tr(\bb\db) \tr(\bb)\\
   &-&3\tr(\ab\bb) \tr(\cb\db)\tr(\bb)-3 \tr(\bb\cb\db)\tr(\ab) \tr(\bb)\\
   &-&3\tr(\ab\bb\db) \tr(\cb)\tr(\bb)+3 \tr(\ab\bb\cb)\tr(\db) \tr(\bb)\\
   &+&3\tr(\bb\cb) \tr(\ab)\tr(\db) \tr(\bb)+3\tr(\ab\bb) \tr(\cb)\tr(\db) \tr(\bb)\\
   &-&6\tr(\bb\db)\tr(\ab\bb\cb)+6\tr(\bb\cb)\tr(\ab\bb\db)+12\tr(\ab\cb\db)\\
   &+&6\tr(\ab\bb)\tr(\bb\cb\db)-6\tr(\cb\db) \tr(\ab)-6\tr(\ab\db) \tr(\cb)\\
   &-&6\tr(\ab\cb) \tr(\db)-6\tr(\ab\bb) \tr(\bb\cb)\tr(\db)+6 \tr(\ab) \tr(\cb)\tr(\db),
\end{eqnarray*}
 
\begin{eqnarray*}
f_3&=&3 \tr(\ab\bb\db)\tr(\cb)^2-3 \tr(\ab\db)\tr(\bb) \tr(\cb)^2\\
   &-&3\tr(\ab\bb) \tr(\db)\tr(\cb)^2+3 \tr(\ab) \tr(\bb)\tr(\db) \tr(\cb)^2\\
   &+&3\tr(\ab\db) \tr(\bb\cb)\tr(\cb)-3 \tr(\ab\cb)\tr(\bb\db) \tr(\cb)\\
   &+&3\tr(\ab\bb) \tr(\cb\db)\tr(\cb)-3 \tr(\bb\cb\db)\tr(\ab) \tr(\cb)\\
   &+&3\tr(\ab\cb\db) \tr(\bb)\tr(\cb)-3 \tr(\cb\db)\tr(\ab) \tr(\bb) \tr(\cb)\\
   &+&3\tr(\ab\bb\cb) \tr(\db)\tr(\cb)-3 \tr(\bb\cb)\tr(\ab) \tr(\db) \tr(\cb)\\
   &-&6\tr(\cb\db)\tr(\ab\bb\cb)-12\tr(\ab\bb\db)-6\tr(\bb\cb)\tr(\ab\cb\db)\\
   &+&6\tr(\ab\cb)\tr(\bb\cb\db)+6\tr(\bb\db) \tr(\ab)+6\tr(\bb\cb) \tr(\cb\db)\tr(\ab)\\
   &+&6 \tr(\ab\db)\tr(\bb)+6 \tr(\ab\bb)\tr(\db)-6 \tr(\ab) \tr(\bb)\tr(\db),
\end{eqnarray*}

\begin{eqnarray*}
f_4&=& -3 \tr(\ab\bb\cb)\tr(\db)^2+3 \tr(\bb\cb)\tr(\ab) \tr(\db)^2\\
   &+&3\tr(\ab\bb) \tr(\cb)\tr(\db)^2-3 \tr(\ab) \tr(\bb)\tr(\cb) \tr(\db)^2\\
   &-&3\tr(\ab\db) \tr(\bb\cb)\tr(\db)+3 \tr(\ab\cb)\tr(\bb\db) \tr(\db)\\
   &-&3\tr(\ab\bb) \tr(\cb\db)\tr(\db)-3 \tr(\bb\cb\db)\tr(\ab) \tr(\db)\\
   &+&3\tr(\ab\cb\db) \tr(\bb)\tr(\db)+3 \tr(\cb\db)\tr(\ab) \tr(\bb) \tr(\db)\\
   &-&3\tr(\ab\bb\db) \tr(\cb)\tr(\db)+3 \tr(\ab\db)\tr(\bb) \tr(\cb) \tr(\db)\\
   &+&12\tr(\ab\bb\cb)+6\tr(\cb\db)\tr(\ab\bb\db)-6\tr(\bb\db)\tr(\ab\cb\db)\\
   &+&6\tr(\ab\db)\tr(\bb\cb\db)-6\tr(\bb\cb) \tr(\ab)-6\tr(\ab\cb) \tr(\bb)\\
   &-&6\tr(\ab\db) \tr(\cb\db)\tr(\bb)-6 \tr(\ab\bb)\tr(\cb)+6 \tr(\ab) \tr(\bb)\tr(\cb),
\end{eqnarray*}
  
\begin{eqnarray*}
  f_5&=&36 \tr(\ab\bb)^2+36\tr(\ab\cb) \tr(\bb\cb)\tr(\ab\bb)\\
     &-&36 \tr(\ab)\tr(\bb) \tr(\ab\bb)-36\tr(\ab\bb\cb) \tr(\cb)\tr(\ab\bb)\\
     &+&36\tr(\ab\cb)^2+36\tr(\bb\cb)^2+36\tr(\ab\bb\cb)^2+36\tr(\ab)^2\\
     &+&36 \tr(\bb)^2+36\tr(\cb)^2-36 \tr(\bb\cb)\tr(\ab\bb\cb)\tr(\ab)\\
     &-&36 \tr(\ab\cb)\tr(\ab\bb\cb)\tr(\bb)-36 \tr(\ab\cb)\tr(\ab) \tr(\cb)\\
     &-&36\tr(\bb\cb) \tr(\bb)\tr(\cb)+36\tr(\ab\bb\cb) \tr(\ab)\tr(\bb) \tr(\cb)-144,
\end{eqnarray*}
 
\begin{eqnarray*}
f_6&=&36 \tr(\ab\cb)^2+36\tr(\ab\db) \tr(\cb\db)\tr(\ab\cb)-36 \tr(\ab)\tr(\cb) \tr(\ab\cb)\\
   &-&36\tr(\ab\cb\db) \tr(\db)\tr(\ab\cb)+36\tr(\ab\db)^2+36\tr(\cb\db)^2\\
   &+&36\tr(\ab\cb\db)^2+36\tr(\ab)^2+36 \tr(\cb)^2+36\tr(\db)^2\\
   &-&36 \tr(\cb\db)\tr(\ab\cb\db)\tr(\ab)-36 \tr(\ab\db)\tr(\ab\cb\db)\tr(\cb)\\
   &-&36 \tr(\ab\db)\tr(\ab) \tr(\db)-36\tr(\cb\db) \tr(\cb)\tr(\db)\\
   &+&36\tr(\ab\cb\db) \tr(\ab)\tr(\cb) \tr(\db)-144,
\end{eqnarray*}
 
\begin{eqnarray*}
f_7&=&36 \tr(\bb\cb)^2+36\tr(\bb\db) \tr(\cb\db)\tr(\bb\cb)\\
   &-&36 \tr(\bb)\tr(\cb) \tr(\bb\cb)-36\tr(\bb\cb\db) \tr(\db)\tr(\bb\cb)\\
   &+&36\tr(\bb\db)^2+36\tr(\cb\db)^2+36\tr(\bb\cb\db)^2+36\tr(\bb)^2\\
   &+&36 \tr(\cb)^2+36\tr(\db)^2-36 \tr(\cb\db)\tr(\bb\cb\db)\tr(\bb)\\
   &-&36 \tr(\bb\db)\tr(\bb\cb\db)\tr(\cb)-36 \tr(\bb\db)\tr(\bb) \tr(\db)\\
   &-&36\tr(\cb\db) \tr(\cb)\tr(\db)+36\tr(\bb\cb\db) \tr(\bb)\tr(\cb) \tr(\db)-144,
\end{eqnarray*}
 
\begin{eqnarray*}
f_8&=&36 \tr(\ab\bb)^2+36\tr(\ab\db) \tr(\bb\db)\tr(\ab\bb)\\
   &-&36 \tr(\ab)\tr(\bb) \tr(\ab\bb)-36\tr(\ab\bb\db) \tr(\db)\tr(\ab\bb)\\
   &+&36\tr(\ab\db)^2+36\tr(\bb\db)^2+36\tr(\ab\bb\db)^2+36\tr(\ab)^2\\
   &+&36 \tr(\bb)^2+36\tr(\db)^2-36 \tr(\bb\db)\tr(\ab\bb\db)\tr(\ab)\\
   &-&36 \tr(\ab\db)\tr(\ab\bb\db)\tr(\bb)-36 \tr(\ab\db)\tr(\ab) \tr(\db)\\
   &-&36\tr(\bb\db) \tr(\bb)\tr(\db)+36\tr(\ab\bb\db) \tr(\ab)\tr(\bb) \tr(\db)-144,
\end{eqnarray*}
 
\begin{eqnarray*}
f_9&=&18 \tr(\bb\db)\tr(\ab\cb)^2-18\tr(\ab\db) \tr(\bb\cb)\tr(\ab\cb)\\
   &-&18\tr(\ab\bb) \tr(\cb\db)\tr(\ab\cb)-18\tr(\ab\cb\db) \tr(\bb)\tr(\ab\cb)\\
   &+&18\tr(\cb\db) \tr(\ab)\tr(\bb) \tr(\ab\cb)-18\tr(\bb\db) \tr(\ab)\tr(\cb) \tr(\ab\cb)\\
   &+&18\tr(\ab\db) \tr(\bb)\tr(\cb) \tr(\ab\cb)-18\tr(\ab\bb\cb) \tr(\db)\tr(\ab\cb)\\
   &+&18\tr(\bb\cb) \tr(\ab)\tr(\db) \tr(\ab\cb)+18\tr(\ab\bb) \tr(\cb)\tr(\db) \tr(\ab\cb)\\
   &-&18\tr(\ab) \tr(\bb) \tr(\cb)\tr(\db) \tr(\ab\cb)+18\tr(\bb\db) \tr(\ab)^2\\
   &+&18\tr(\bb\cb) \tr(\cb\db)\tr(\ab)^2+18 \tr(\ab\bb)\tr(\ab\db) \tr(\cb)^2\\
   &+&18\tr(\bb\db) \tr(\cb)^2-18\tr(\ab\db) \tr(\ab)\tr(\bb) \tr(\cb)^2\\
   &-&36\tr(\ab\bb)\tr(\ab\db)-72\tr(\bb\db)-36\tr(\bb\cb)\tr(\cb\db)\\
   &+&36\tr(\ab\bb\cb)\tr(\ab\cb\db)-18\tr(\cb\db)\tr(\ab\bb\cb)\tr(\ab)\\
   &-&18 \tr(\bb\cb)\tr(\ab\cb\db)\tr(\ab)+18 \tr(\ab\db)\tr(\ab) \tr(\bb)\\
   &-&18\tr(\ab\db)\tr(\ab\bb\cb)\tr(\cb)-18 \tr(\ab\bb)\tr(\ab\cb\db)\tr(\cb)\\
   &+&18 \tr(\ab\db)\tr(\bb\cb) \tr(\ab)\tr(\cb)+18 \tr(\ab\bb)\tr(\cb\db) \tr(\ab)\tr(\cb)\\
   &-&18 \tr(\cb\db)\tr(\ab)^2 \tr(\bb)\tr(\cb)+18 \tr(\cb\db)\tr(\bb) \tr(\cb)\\
   &+&18\tr(\ab\cb\db) \tr(\ab)\tr(\bb) \tr(\cb)-18\tr(\ab\bb) \tr(\ab)\tr(\cb)^2 \tr(\db)\\
   &+&18\tr(\ab)^2 \tr(\bb) \tr(\cb)^2\tr(\db)-18 \tr(\bb)\tr(\cb)^2 \tr(\db)\\
   &+&18\tr(\ab\bb) \tr(\ab)\tr(\db)-18 \tr(\ab)^2\tr(\bb) \tr(\db)+36 \tr(\bb)\tr(\db)\\
   &-&18 \tr(\bb\cb)\tr(\ab)^2 \tr(\cb)\tr(\db)+18 \tr(\bb\cb)\tr(\cb) \tr(\db)\\
   &+&18\tr(\ab\bb\cb) \tr(\ab)\tr(\cb) \tr(\db),
\end{eqnarray*}
 
\begin{eqnarray*} 
f_{10}&=&-18 \tr(\cb\db)\tr(\ab\bb)^2+18 \tr(\cb)\tr(\db) \tr(\ab\bb)^2\\
      &+&18\tr(\ab\db) \tr(\bb\cb)\tr(\ab\bb)+18\tr(\ab\cb) \tr(\bb\db)\tr(\ab\bb)\\
      &+&18\tr(\cb\db) \tr(\ab)\tr(\bb) \tr(\ab\bb)-18\tr(\ab\bb\db) \tr(\cb)\tr(\ab\bb)\\
      &-&18\tr(\ab\bb\cb) \tr(\db)\tr(\ab\bb)-18 \tr(\ab)\tr(\bb) \tr(\cb) \tr(\db)\tr(\ab\bb)\\
      &-&18\tr(\cb\db) \tr(\ab)^2-18\tr(\cb\db) \tr(\bb)^2+36\tr(\ab\cb)\tr(\ab\db)\\
      &+&36\tr(\bb\cb)\tr(\bb\db)+72\tr(\cb\db)+36\tr(\ab\bb\cb)\tr(\ab\bb\db)\\
      &-&18\tr(\bb\db)\tr(\ab\bb\cb)\tr(\ab)-18 \tr(\bb\cb)\tr(\ab\bb\db)\tr(\ab)\\
      &-&18 \tr(\ab\db)\tr(\ab\bb\cb)\tr(\bb)-18 \tr(\ab\cb)\tr(\ab\bb\db)\tr(\bb)\\
      &-&18 \tr(\ab\db)\tr(\ab) \tr(\cb)-18\tr(\bb\db) \tr(\bb)\tr(\cb)\\
      &+&18\tr(\ab\bb\db) \tr(\ab)\tr(\bb) \tr(\cb)-18\tr(\ab\cb) \tr(\ab)\tr(\db)\\
      &-&18 \tr(\bb\cb)\tr(\bb) \tr(\db)+18\tr(\ab\bb\cb) \tr(\ab)\tr(\bb) \tr(\db)\\
      &+&18\tr(\ab)^2 \tr(\cb)\tr(\db)+18 \tr(\bb)^2\tr(\cb) \tr(\db)\\
      &-&36 \tr(\cb)\tr(\db),
\end{eqnarray*}
 
\begin{eqnarray*}
f_{11}&=&-18 \tr(\ab\db)\tr(\bb\cb)^2+18 \tr(\ab)\tr(\db) \tr(\bb\cb)^2\\
      &+&18\tr(\ab\cb) \tr(\bb\db)\tr(\bb\cb)+18\tr(\ab\bb) \tr(\cb\db)\tr(\bb\cb)\\
      &-&18\tr(\bb\cb\db) \tr(\ab)\tr(\bb\cb)+18\tr(\ab\db) \tr(\bb)\tr(\cb) \tr(\bb\cb)\\
      &-&18\tr(\ab\bb\cb) \tr(\db)\tr(\bb\cb)-18 \tr(\ab)\tr(\bb) \tr(\cb) \tr(\db)\tr(\bb\cb)\\
      &-&18\tr(\ab\db) \tr(\bb)^2-18\tr(\ab\db) \tr(\cb)^2+72\tr(\ab\db)\\
      &+&36\tr(\ab\bb)\tr(\bb\db)+36\tr(\ab\cb)\tr(\cb\db)+36\tr(\ab\bb\cb)\tr(\bb\cb\db)\\
      &-&18\tr(\cb\db)\tr(\ab\bb\cb)\tr(\bb)-18 \tr(\ab\cb)\tr(\bb\cb\db)\tr(\bb)\\
      &-&18 \tr(\bb\db)\tr(\ab) \tr(\bb)-18\tr(\bb\db)\tr(\ab\bb\cb)\tr(\cb)\\
      &-&18 \tr(\ab\bb)\tr(\bb\cb\db)\tr(\cb)-18 \tr(\cb\db)\tr(\ab) \tr(\cb)\\
      &+&18\tr(\bb\cb\db) \tr(\ab)\tr(\bb) \tr(\cb)+18 \tr(\ab)\tr(\bb)^2 \tr(\db)\\
      &+&18\tr(\ab) \tr(\cb)^2\tr(\db)-36 \tr(\ab)\tr(\db)-18 \tr(\ab\bb)\tr(\bb) \tr(\db)\\
      &-&18\tr(\ab\cb) \tr(\cb)\tr(\db)+18\tr(\ab\bb\cb) \tr(\bb)\tr(\cb) \tr(\db),
\end{eqnarray*}

\begin{eqnarray*}
f_{12}&=&-18 \tr(\bb\cb)\tr(\ab\db)^2+18 \tr(\bb)\tr(\cb) \tr(\ab\db)^2\\
      &+&18\tr(\ab\cb) \tr(\bb\db)\tr(\ab\db)+18\tr(\ab\bb) \tr(\cb\db)\tr(\ab\db)\\
      &-&18\tr(\ab\cb\db) \tr(\bb)\tr(\ab\db)-18\tr(\ab\bb\db) \tr(\cb)\tr(\ab\db)\\
      &+&18\tr(\bb\cb) \tr(\ab)\tr(\db) \tr(\ab\db)-18\tr(\bb\cb) \tr(\ab)^2\\
      &-&18\tr(\ab) \tr(\bb) \tr(\cb)\tr(\db) \tr(\ab\db)-18\tr(\bb\cb) \tr(\db)^2\\
      &+&18\tr(\bb) \tr(\cb)\tr(\db)^2+36 \tr(\ab\bb)\tr(\ab\cb)+72\tr(\bb\cb)\\
      &+&36\tr(\bb\db)\tr(\cb\db)+36\tr(\ab\bb\db)\tr(\ab\cb\db)\\
      &-&18\tr(\cb\db)\tr(\ab\bb\db)\tr(\ab)-18 \tr(\bb\db)\tr(\ab\cb\db)\tr(\ab)\\
      &-&18 \tr(\ab\cb)\tr(\ab) \tr(\bb)-18\tr(\ab\bb) \tr(\ab)\tr(\cb)\\
      &+&18 \tr(\ab)^2\tr(\bb) \tr(\cb)-36 \tr(\bb)\tr(\cb)\\
      &-&18 \tr(\ab\cb)\tr(\ab\bb\db)\tr(\db)-18 \tr(\ab\bb)\tr(\ab\cb\db)\tr(\db)\\
      &-&18 \tr(\cb\db)\tr(\bb) \tr(\db)+18\tr(\ab\cb\db) \tr(\ab)\tr(\bb) \tr(\db)\\
      &-&18\tr(\bb\db) \tr(\cb)\tr(\db)+18\tr(\ab\bb\db) \tr(\ab)\tr(\cb) \tr(\db),
\end{eqnarray*}
 
\begin{eqnarray*}
f_{13}&=&18 \tr(\ab\cb)\tr(\bb\db)^2-18\tr(\ab\db) \tr(\bb\cb)\tr(\bb\db)\\
      &-&18\tr(\ab\bb) \tr(\cb\db)\tr(\bb\db)-18\tr(\bb\cb\db) \tr(\ab)\tr(\bb\db)\\
      &+&18\tr(\cb\db) \tr(\ab)\tr(\bb) \tr(\bb\db)-18\tr(\ab\bb\db) \tr(\cb)\tr(\bb\db)\\
      &+&18\tr(\ab\db) \tr(\bb)\tr(\cb) \tr(\bb\db)+18\tr(\bb\cb) \tr(\ab)\tr(\db) \tr(\bb\db)\\
      &-&18\tr(\ab\cb) \tr(\bb)\tr(\db) \tr(\bb\db)+18\tr(\ab\bb) \tr(\cb)\tr(\db) \tr(\bb\db)\\
      &-&18\tr(\ab) \tr(\bb) \tr(\cb)\tr(\db) \tr(\bb\db)+18\tr(\ab\cb) \tr(\bb)^2\\
      &+&18\tr(\ab\db) \tr(\cb\db)\tr(\bb)^2+18 \tr(\ab\cb)\tr(\db)^2\\
      &+&18 \tr(\ab\bb)\tr(\bb\cb) \tr(\db)^2-18\tr(\bb\cb) \tr(\ab)\tr(\bb) \tr(\db)^2\\
      &+&18\tr(\ab) \tr(\bb)^2 \tr(\cb)\tr(\db)^2-18 \tr(\ab)\tr(\cb) \tr(\db)^2\\
      &-&18\tr(\ab\bb) \tr(\bb)\tr(\cb) \tr(\db)^2-72\tr(\ab\cb)-36\tr(\ab\bb)\tr(\bb\cb)\\
      &-&36\tr(\ab\db)\tr(\cb\db)+36\tr(\ab\bb\db)\tr(\bb\cb\db)\\
      &-&18\tr(\cb\db)\tr(\ab\bb\db)\tr(\bb)-18 \tr(\ab\db)\tr(\bb\cb\db)\tr(\bb)\\
      &+&18 \tr(\bb\cb)\tr(\ab) \tr(\bb)-18 \tr(\ab)\tr(\bb)^2 \tr(\cb)\\
      &+&36\tr(\ab) \tr(\cb)+18\tr(\ab\bb) \tr(\bb)\tr(\cb)\\
      &-&18 \tr(\cb\db)\tr(\ab) \tr(\bb)^2\tr(\db)-18 \tr(\bb\cb)\tr(\ab\bb\db)\tr(\db)\\
      &-&18 \tr(\ab\bb)\tr(\bb\cb\db)\tr(\db)+18 \tr(\cb\db)\tr(\ab) \tr(\db)\\
      &+&18\tr(\ab\db) \tr(\bb\cb)\tr(\bb) \tr(\db)+18\tr(\ab\bb) \tr(\cb\db)\tr(\bb) \tr(\db)\\
      &+&18\tr(\bb\cb\db) \tr(\ab)\tr(\bb) \tr(\db)-18\tr(\ab\db) \tr(\bb)^2\tr(\cb) \tr(\db)\\
      &+&18\tr(\ab\db) \tr(\cb)\tr(\db)+18\tr(\ab\bb\db) \tr(\bb)\tr(\cb) \tr(\db),
\end{eqnarray*}
 
\begin{eqnarray*}
f_{14}&=&-18 \tr(\ab\bb)\tr(\cb\db)^2+18 \tr(\ab)\tr(\bb) \tr(\cb\db)^2\\
      &+&18\tr(\ab\db) \tr(\bb\cb)\tr(\cb\db)+18\tr(\ab\cb) \tr(\bb\db)\tr(\cb\db)\\
      &-&18\tr(\bb\cb\db) \tr(\ab)\tr(\cb\db)-18\tr(\ab\cb\db) \tr(\bb)\tr(\cb\db)\\
      &+&18\tr(\ab\bb) \tr(\cb)\tr(\db) \tr(\cb\db)-18\tr(\ab\bb) \tr(\cb)^2\\
      &-&18\tr(\ab) \tr(\bb) \tr(\cb)\tr(\db) \tr(\cb\db)+18\tr(\ab) \tr(\bb)\tr(\cb)^2\\
      &-&18 \tr(\ab\bb)\tr(\db)^2+18 \tr(\ab)\tr(\bb) \tr(\db)^2+72\tr(\ab\bb)\\
      &+&36\tr(\ab\cb)\tr(\bb\cb)+36\tr(\ab\db)\tr(\bb\db)+36\tr(\ab\cb\db)\tr(\bb\cb\db)\\
      &-&36\tr(\ab) \tr(\bb)-18\tr(\bb\db)\tr(\ab\cb\db)\tr(\cb)\\
      &-&18 \tr(\ab\db)\tr(\bb\cb\db)\tr(\cb)-18 \tr(\bb\cb)\tr(\ab) \tr(\cb)\\
      &-&18\tr(\ab\cb) \tr(\bb)\tr(\cb)-18 \tr(\bb\cb)\tr(\ab\cb\db)\tr(\db)\\
      &-&18 \tr(\ab\cb)\tr(\bb\cb\db)\tr(\db)-18 \tr(\bb\db)\tr(\ab) \tr(\db)\\
      &-&18\tr(\ab\db) \tr(\bb)\tr(\db)+18\tr(\bb\cb\db) \tr(\ab)\tr(\cb) \tr(\db)\\
      &+&18\tr(\ab\cb\db) \tr(\bb)\tr(\cb) \tr(\db).
\end{eqnarray*}
}

\newcommand{\etalchar}[1]{$^{#1}$}
\def\cdprime{$''$} \def\cdprime{$''$} \def\cprime{$'$} \def\cprime{$'$}
  \def\cprime{$'$} \def\cprime{$'$}

\end{document}